\documentclass{amsart}

\usepackage[all]{xy}

\theoremstyle{plain}
\newtheorem{thm}{Theorem}[section]
\newtheorem{prop}[thm]{Proposition}
\newtheorem{lemma}[thm]{Lemma}
\newtheorem{cor}[thm]{Corollary}
\newtheorem{question}[thm]{Question}

\theoremstyle{definition}
\newtheorem{dfn}[thm]{Definition}

\newtheorem{hypo}{Hypothesis}

\theoremstyle{remark}
\newtheorem{rem}[thm]{Remark}

\newcommand{\HH}{\mathrm{H}}

\newcommand{\semid}{\unitlength.47cm
 \begin{picture}(.7,.6)
   \put(0,.05){$\times$}
   \put(.47,.04){\line(0,1){.39}}
 \end{picture}}

\begin{document}

\title[Dihedral blocks with two simple modules]{Universal deformation rings and dihedral
blocks with two simple modules}

\author{Frauke M. Bleher}
\address{F.B.: Department of Mathematics\\University of Iowa\\
Iowa City, IA 52242-1419, U.S.A.}
\email{frauke-bleher@uiowa.edu}
\thanks{The first author was supported in part by  
NSF Grant DMS06-51332.}
\author{Giovanna LLosent}
\address{G.L.: Department of Mathematics\\CSU
San Bernardino, CA 92407-2397, U.S.A.}
\email{gllosent@csusb.edu}
\author{Jennifer B. Schaefer}
\address{J.S.: Department of Mathematics and Computer Science\\
Dickinson College\\ Carlisle, PA 17013, U.S.A.}
\email{schaefje@dickinson.edu}

\begin{abstract}
Let $k$ be an algebraically closed field of characteristic $2$, and let $W$ be the ring of infinite Witt 
vectors over $k$. Suppose $G$ is a finite group and  $B$ is a block of $kG$ with a dihedral defect 
group $D$ such that there are precisely two isomorphism classes of simple $B$-modules.
We determine the universal deformation ring $R(G,V)$ for every finitely generated $kG$-module $V$ 
which belongs to $B$ and whose stable endomorphism ring is isomorphic to $k$. 
The description by Erdmann of the quiver and relations of the basic algebra of $B$ is usually only 
determined up to a certain parameter $c$ which is either $0$ or $1$. 
We show that $R(G,V)$ is isomorphic to a subquotient ring of $WD$ for all $V$ as above 
if and only if $c=0$, giving an answer to a question raised by the first author and Chinburg in this case.
Moreover, we prove that $c=0$ if and only if  $B$ is Morita equivalent to a principal 
block.
\end{abstract}

\maketitle


\section{Introduction}
\label{s:intro}

Let $k$ be an algebraically closed field of characteristic $p>0$ and let $W=W(k)$ be the ring of infinite 
Witt vectors over $k$. Let $G$ be a finite group, and suppose $V$ is a finitely generated $kG$-module
whose stable endomorphism ring is isomorphic to $k$. It was shown in \cite{bc} that $V$ has a universal 
deformation ring $R(G,V)$ which is universal with respect to deformations of $V$ over complete local 
commutative Noetherian rings with residue field $k$ (see Section \ref{s:prelim}). 
In \cite{bc} (resp. \cite{bl}), the isomorphism types of $R(G,V)$ were determined for $V$ belonging
to a cyclic block (resp. to a block with Klein four defect groups). In \cite{diloc,3sim}, the rings
$R(G,V)$ were determined for $V$ belonging to various tame blocks with dihedral defect 
groups of order
at least 8 with one or three isomorphism classes of simple modules. 
Using the classification of all
groups with dihedral Sylow $2$-subgroups by Gorenstein and Walter in \cite{gowa}, it follows that 
these blocks include in particular all those blocks that are Morita equivalent to a principal block with
dihedral defect groups and that have one or three isomorphism classes of simple modules. 
By \cite[Thm. 2]{brauer2},
a block of $kG$ with dihedral defect groups has at most three simple modules up to isomorphism.
Hence it remains to study the case of dihedral blocks with precisely two isomorphism classes of simple
modules, and this is the case we consider in the present paper.
The key tools used to determine the universal deformation rings in all the above cases are results
from modular and ordinary representation theory due to Brauer, Erdmann \cite{erd}, 
Linckelmann \cite{linckel,linckel1}, Carlson-Th\'{e}venaz \cite{carl2}, and others.

The main motivation for studying universal deformation rings for finite groups is that this case helps
understand ring theoretic properties of universal deformation rings for profinite groups $\Gamma$.
The latter have become an important tool in number theory, in particular if $\Gamma$ is a
profinite Galois group (see e.g. \cite{cornell}, \cite{wiles,taywiles}, \cite{breuil} and their references).
In \cite{lendesmit}, de Smit and Lenstra showed
that if $\Gamma$ is an arbitrary profinite group and $V$ is a finite
dimensional vector space over $k$ with a continuous $\Gamma$-action which has a universal
deformation ring  $R(\Gamma,V)$, then $R(\Gamma,V)$ is the inverse limit of the universal 
deformation rings $R(G,V)$ when $G$ runs over all finite discrete quotients of $\Gamma$ through 
which the $\Gamma$-action on $V$ factors. Thus to answer questions about the ring structure of 
$R(\Gamma,V)$, it is natural to first consider the case when $\Gamma=G$ is finite.
When determining $R(G,V)$, the main advantage is that one can make use of powerful techniques 
that are not available for arbitrary profinite groups $\Gamma$,
such as decomposition matrices, Auslander-Reiten theory and the Green correspondence.

When studying cyclic blocks in \cite{bc}, the first author and Chinburg raised the following
fundamental question about the structure of universal deformation rings for finite groups.

\begin{question}
\label{qu:main}
Let $B$ be a block of $kG$ with a defect group $D$, and suppose $V$ is a finitely generated 
$kG$-module whose stable endomorphism ring $\underline{\mathrm{End}}_{kG}(V)$
is isomorphic to $k$ such that the unique 
$($up to isomorphism$)$ non-projective indecomposable summand of $V$ belongs to $B$. Is the 
universal deformation ring $R(G,V)$ of $V$ isomorphic to a subquotient ring of the group ring $WD$?
\end{question}

The results in \cite{bl,diloc,3sim,bc,blello} show that Question \ref{qu:main} has a positive answer for all 
the blocks $B$ considered in these papers. In particular, this is true if $B$ has cyclic or Klein four defect 
groups or if $B$ is Morita equivalent to a principal block with dihedral defect groups and one or three 
isomorphism classes of simple modules. Another important question which was raised by Flach 
\cite{flach} is whether universal deformation rings are always complete intersections. One of the first
examples that this is not true in general was given in \cite{bc4.9,bc5}. Namely, if $k$ has 
characteristic $2$, $S_4$ is the symmetric group on $4$ letters and $E$ is the unique (up to 
isomorphism) non-trivial simple $kS_4$-module, then $R(S_4,E)\cong W[t]/(t^2,2t)$. 
Subsequently, it was shown in \cite[Cor. 5.1.2]{3sim} that if $B$ is Morita equivalent to a principal 
block with dihedral defect groups of order at least 8 and with three isomorphism classes of simple 
modules, then there are infinitely many indecomposable $B$-modules $V$ whose stable 
endomorphism ring is isomorphic to $k$ and whose universal deformation ring $R(G,V)$ is not a 
complete intersection.

Suppose now that $k$ has characteristic $2$ and $B$ is a block of $kG$ with dihedral defect groups
of order $2^d\ge 8$ and with precisely two isomorphism classes of simple $B$-modules. In
\cite[Chaps. VI, IX and p. 294--295]{erd}, Erdmann gave a list of all possible quivers and relations 
which determine the basic algebra of such a block $B$ up to isomorphism.
Besides the defect $d$ of $B$, she needed an additional parameter $c\in\{0,1\}$ to describe the 
isomorphism type of the basic algebra of $B$. Given $B$, it is usually difficult to determine 
whether $c=0$ or $c=1$. 
So far, only a few explicit cases are known where $c$ has been 
determined, such as the principal $2$-modular block of the symmetric group $S_4$ 
(see \cite[Cor. V.2.5.1]{erd}), the principal $2$-modular blocks of certain quotients of  the general 
unitary group $\mathrm{GU}_2(\mathbb{F}_q)$ when $q\equiv 3 \mod 4$ (see 
\cite[Sect.  1.5]{erdlater}), or the principal $2$-modular block of the projective general
linear group $\mathrm{PGL}_2(\mathbb{F}_q)$ when $q$ is any odd prime power (see
\cite[Cor. 4]{parameter}). In all these cases, $c$ turns out to be zero.
Moreover, it was proved in \cite[Thm. 2]{parameter} that $c$ is zero for all blocks $B$
for which there exists a central extension $\hat{G}$ of $G$ by a group of order $2$ 
together with a block $\hat{B}$ of $k\hat{G}$ with generalized quaternion defect groups 
such that $B$ is contained in the image of the natural surjection $k\hat{G}\to kG$.
Using the results in \cite{parameter}  together with the classification of all groups with dihedral Sylow
$2$-subgroups by Gorenstein and Walter in \cite{gowa}, we will first prove the following result:

\begin{thm}
\label{thm:extraout}
Suppose $k$ has characteristic $2$.
Let $B$ be a block of $kG$ with dihedral defect groups of order at least $8$ such
that there are precisely two isomorphism classes of simple $B$-modules. Then $B$ is Morita
equivalent to a principal block if and only if the parameter $c$ occurring in Erdmann's description
in \cite{erd} of the quiver and relations of the basic algebra of $B$ is $c=0$.
\end{thm}

At this point, it is still open whether the parameter $c$ can take the value $c=1$ for non-principal 
blocks $B$.

Let $\Omega$ be the syzygy functor, which defines a self-equivalence on the stable module category
of $kG$. Recall that if $V_0$ is a finitely generated indecomposable non-projective $kG$-module, 
then $\Omega(V_0)$ is defined to be the kernel of a projective cover $P_{V_0}\to V_0$. We say 
$V_0$ has $\Omega$-period $\ell$ if $\ell$ is the smallest positive integer with 
$\Omega^\ell(V_0)\cong V_0$.

We will determine the universal deformation rings of all finitely generated $kG$-modules which
belong to an arbitrary block $B$ of $kG$ with dihedral defect groups and precisely two isomorphism
classes of simple modules.
A summary of our main results is as follows. The precise 
statements can be found in Propositions \ref{prop:stablend}, \ref{prop:udrc0}, \ref{prop:3tube}, 
\ref{prop:omegaorbit} and  \ref{prop:onetubes}. 

\begin{thm}
\label{thm:bigmain}
Suppose $k$ has characteristic $2$.
Let $B$ be a block of $kG$ with a dihedral defect group $D$ of order $2^d$ where $d\ge 3$, such
that there are precisely two isomorphism classes of simple $B$-modules. 
Let $c$ be the parameter occurring in Erdmann's description in \cite{erd} of the quiver and relations 
of the basic algebra of $B$.
Let $V$ be a finitely generated $B$-module whose stable endomorphism ring  
is isomorphic to $k$ and whose universal deformation ring is $R(G,V)$. Let $V_0$ be the unique 
$($up to isomorphism$)$ non-projective indecomposable direct summand of $V$.
Let $\mathfrak{C}$ be the component of the stable Auslander-Reiten quiver of $B$
to which $V_0$ belongs. Then one of the following $4$ mutually exclusive cases occurs:
\begin{enumerate}
\item[(i)] The module $V_0$ is not $\Omega$-periodic, the stable endomorphism ring of every module in
$\mathfrak{C}$ is isomorphic to $k$ and $R(G,V)$ is isomorphic to a $W$-subalgebra of 
$W[\mathbb{Z}/2]$. Moreover, if $B$ is a principal block, then $R(G,V)\cong W[\mathbb{Z}/2]$.
\item[(ii)] The module $V_0$ is not $\Omega$-periodic, the only modules in $\mathfrak{C}$ whose stable
endomorphism rings are isomorphic to $k$ are the modules in the $\Omega$-orbit of $V_0$ 
and $R(G,V)\cong W[[t]]/(t\, p_d(t),2\, p_d(t))$ for a certain monic polynomial 
$p_d(t)\in W[t]$ of degree $2^{d-2}-1$ whose non-leading coefficients are all divisible by $2$.
\item[(iii)] The module $V_0$ has $\Omega$-period $3$, the only modules in $\mathfrak{C}$ whose stable
endomorphism rings are isomorphic to $k$ are the modules at the boundary of $\mathfrak{C}$
and $R(G,V)\cong k$.
\item[(iv)] The module $V_0$ has $\Omega$-period $2$, the only module in 
$\mathfrak{C}$ whose stable endomorphism ring is isomorphic to $k$ is the module at the 
boundary of $\mathfrak{C}$ and $R(G,V)\cong W[[t]]/(2\, f_{V_0}(t))$ for a certain power series 
$f_{V_0}(t)\in W[[t]]$.
\end{enumerate}
In cases $(i) - (iii)$, $R(G,V)$ is isomorphic to a subquotient ring of $WD$. 
In case $(ii)$, $R(G,V)$ is not a complete intersection.
Cases $(i) - (iii)$ occur for both $c=0$ and $c=0$, whereas case $(iv)$ occurs if and only if $c=1$.
If case $(iv)$ occurs, then there are infinitely many indecomposable $B$-modules 
with $\Omega$-period $2$ whose stable endomorphism rings are isomorphic to $k$.
\end{thm}

Using Theorem \ref{thm:extraout}, we obtain the following consequence of Theorem 
\ref{thm:bigmain}.

\begin{cor}
\label{cor:littlecor}
Let $B$ and $c$ be as in Theorem $\ref{thm:bigmain}$. Then the following statements are equivalent:
\begin{enumerate}
\item[(i)] The block $B$ is Morita equivalent to a principal block.
\item[(ii)] The parameter $c$ is zero.
\item[(iii)] Case $($iv$)$ of Theorem $\ref{thm:bigmain}$ does not occur for $B$.
\item[(iv)] Question $\ref{qu:main}$ has a positive answer for $B$. 
\end{enumerate}
\end{cor}

Corollary \ref{cor:littlecor} together with the results in \cite{bl,diloc,3sim} lead to the following answer to Question \ref{qu:main}
for principal blocks with dihedral defect groups. 
Note that all blocks with Klein four defect groups are Morita equivalent
to principal blocks (see e.g. \cite[\S 6.6]{ben}).

\begin{thm}
\label{cor:main}
Suppose $k$ has characteristic $2$, and suppose $B$ is a block of $kG$ with dihedral defect groups of
order at least $4$ which is Morita equivalent to a principal block. Then Question $\ref{qu:main}$
has a positive answer for $B$.
\end{thm}

The paper is organized as follows. In Section  \ref{s:prelim}, we review the basic definitions and results
concerning universal deformation rings of modules for finite groups. We then let $B$ be a 
$2$-modular block with dihedral defect groups such that there are precisely two isomorphism 
classes of simple $B$-modules. In Section  \ref{s:dihedralblocks}, we 
provide the quiver and relations for the basic algebra of $B$ as given in \cite{erd},
and we describe the ordinary characters
belonging to $B$ as given in \cite{brauer2}. 
In Section \ref{s:principal}, we prove Theorem \ref{thm:extraout}.
In Section \ref{s:stablend}, we determine all finitely generated
$B$-modules whose stable endomorphism rings are isomorphic to $k$. In Section \ref{s:udr}, we prove 
Theorem \ref{thm:bigmain}.
In Section \ref{s:stringband}, we provide some background on the representation theory of the
basic algebra of $B$.


\section{Preliminaries}
\label{s:prelim}

Let $k$ be an algebraically closed field of characteristic $p>0$, let $W$ be the ring of infinite Witt 
vectors over $k$ and let $F$ be the fraction field of $W$. Let ${\mathcal{C}}$ be the category of 
all complete local commutative Noetherian rings with residue field $k$. The morphisms in 
${\mathcal{C}}$ are continuous $W$-algebra homomorphisms which induce the identity map on $k$. 

Suppose $G$ is a finite group and $V$ is a finitely generated $kG$-module. 
A lift of $V$ over an object $R$ in ${\mathcal{C}}$ is a finitely generated $RG$-module $M$ which
is free over $R$ together with a $kG$-module isomorphism $\phi: k\otimes_R M\to V$. 
We denote such a lift by $(M,\phi)$. Two lifts $(M,\phi)$ and $(M',\phi')$ of $V$ over $R$ are said to be 
isomorphic if there is an $RG$-module isomorphism $f:M\to M'$ such that $\phi'\circ(k\otimes_Rf)=
\phi$. The isomorphism class $[M,\phi]$ of a lift $(M,\phi)$ of $V$ 
over $R$ is called a deformation of $V$ over $R$, and the set of such deformations is denoted by 
$\mathrm{Def}_G(V,R)$. The deformation functor ${F}_V:{\mathcal{C}} \to \mathrm{Sets}$
is defined to be the covariant functor which sends an object $R$ in ${\mathcal{C}}$ to 
$\mathrm{Def}_G(V,R)$ and a morphism $\alpha:R\to R'$ in $\mathcal{C}$ to the set map
$F_V(\alpha):\mathrm{Def}_G(V,R)\to \mathrm{Def}_G(V,R')$ such that $F_V([M,\phi]) = 
[R'\otimes_{R,\alpha}M,\phi_\alpha]$ where $\phi_\alpha$ is the composition
$k\otimes_{R'}(R'\otimes_{R,\alpha}M)\cong k\otimes_R M \xrightarrow{\phi} V$.

In case there exists an object $R(G,V)$ in ${\mathcal{C}}$ and a lift $(U(G,V),\phi_U)$ of $V$ over 
$R(G,V)$ such that for each $R$ in ${\mathcal{C}}$ and for each lift $(M,\phi)$ of $V$ over $R$ 
there exists a unique morphism $\alpha:R(G,V)\to R$ in ${\mathcal{C}}$ such that 
$[M,\phi]=F_V(\alpha)([U(G,V),\phi_U])$, then $R(G,V)$ is called the universal deformation ring of 
$V$ and the deformation $[U(G,V),\phi_U]$ is called the universal deformation of $V$. 
In other words, $R(G,V)$ represents the functor ${F}_V$ in the sense that ${F}_V$ is naturally 
isomorphic to $\mathrm{Hom}_{{\mathcal{C}}}(R(G,V),-)$. For more information on deformation 
rings see \cite{lendesmit} and \cite{maz1}.

\begin{rem}
\label{rem:deformations}
The above definition of deformations differs from the definition used in \cite{bl,bc}.
Namely, given a lift $(M,\phi)$ of $V$ over an object $R$ in $\mathcal{C}$, in \cite{bl,bc}
the isomorphism class of $M$ as an $RG$-module was called a 
deformation of $V$ over $R$, without taking into account the specific isomorphism 
$\phi:k\otimes_RM\to V$. 
In general, a deformation of $V$ over $R$ according to the latter definition
identifies more lifts than a deformation of $V$ over $R$
that respects the isomorphism $\phi$ of a representative $(M,\phi)$.
However, if the stable
endomorphism ring $\underline{\mathrm{End}}_{kG}(V)$ is isomorphic to $k$, these two definitions
of deformations coincide (see e.g. \cite[Remark 2.1]{quaternion}). 
\end{rem}

The following result was proved in \cite{bc}.
As above, $\Omega$ denotes the syzygy functor.

\begin{prop}
\label{prop:stablendudr}
{\rm \cite[Prop. 2.1 and Cors. 2.5 and 2.8]{bc}}
Suppose $V$ is a finitely generated $kG$-module whose stable endomorphism ring 
$\underline{\mathrm{End}}_{kG}(V)$ is isomorphic to $k$.  Then
$V$ has  a universal deformation ring $R(G,V)$. Moreover:
\begin{itemize}
\item[(i)] If $V_0$ is the unique $($up to isomophism$)$ non-projective indecomposable summand 
of $V$, then $\underline{\mathrm{End}}_{kG}(V_0)\cong k$ and $R(G,V)\cong R(G,V_0)$.
\item[(ii)] We have
$\underline{\mathrm{End}}_{kG}(\Omega(V))\cong k$ and $R(G,V)\cong R(G,\Omega(V))$.
\end{itemize}
\end{prop}


\section{Dihedral blocks with two simple modules}
\label{s:dihedralblocks}

In \cite{brauer2}, Brauer  proved that a $2$-modular
block with dihedral defect groups contains at most three 
simple modules up to isomorphism. 
For the remainder of the paper, we make the following assumptions.

\begin{hypo}
\label{hyp:throughout}
Let $k$ be an algebraically closed field of characteristic $2$, let $W$ be the ring of infinite
Witt vectors over $k$ and let $F$ be the fraction field of $W$. Let $d\ge 3$ be a fixed integer.
Suppose $G$ is a finite group and $B$ is a block of $kG$ having a dihedral defect group $D$ 
of order $2^d$ such that there are precisely two isomorphism classes of simple $B$-modules.
\end{hypo}


\subsection{Basic algebras of dihedral blocks with two simple modules}
\label{ss:basicalgebras}

Under the assumptions of Hypothesis \ref{hyp:throughout}, it follows from 
\cite[Chaps. VI, IX and p. 294--295]{erd} that there exist
$i\in\{1,2\}$ and $c\in\{0,1\}$ such that the basic algebra of $B$ is isomorphic to the symmetric 
$k$-algebra $\Lambda_{i,c}$, as defined in Figure \ref{fig:basic}.
\begin{figure}[ht] \hrule \caption{\label{fig:basic} The basic algebras 
$\Lambda_{1,c}=k\,Q_1/I_{1,c}$ and $\Lambda_{2,c}=k\,Q_2/I_{2,c}$.}
$$\begin{array}{rl}
\raisebox{-2ex}{$Q_1\quad=$}&\xymatrix @R=-.2pc {
0&1\\
\ar@(ul,dl)_{\alpha} \bullet \ar@<.8ex>[r]^{\beta} &\bullet\ar@<.9ex>[l]^{\gamma}}
\quad\raisebox{-2ex}{and}\\[5ex]
I_{1,c}\quad=&\langle \beta\gamma, \alpha^2-c\,(\gamma\beta\alpha)^{2^{d-2}},
(\gamma\beta\alpha)^{2^{d-2}}-(\alpha\gamma\beta)^{2^{d-2}}\rangle,\\[3ex]
\raisebox{-2ex}{$Q_2\quad=$}&\xymatrix @R=-.2pc {
0&1\\
\ar@(ul,dl)_{\alpha} \bullet \ar@<.8ex>[r]^{\beta} &\bullet\ar@<.9ex>[l]^{\gamma}
\ar@(ur,dr)^{\eta}}\quad\raisebox{-2ex}{and}\\[5ex]
I_{2,c}\quad=& \langle \eta\beta,\gamma\eta,\beta\gamma, \alpha^2-c\,\gamma\beta\alpha,
\gamma\beta\alpha-\alpha\gamma\beta,\eta^{2^{d-2}}-\beta\alpha\gamma\rangle.
\end{array}$$
\vspace{2ex}
\hrule
\end{figure}
The corresponding decomposition matrices are given in Figure \ref{fig:decomp}.
\begin{figure}[ht] \hrule \caption{\label{fig:decomp} The decomposition matrix for a block $B$
of $kG$ that is Morita equivalent to $\Lambda_{1,c}$ (resp. $\Lambda_{2,c}$).}
$$\begin{array}{ccc}
&\begin{array}{c@{}c}\varphi_0\,&\,\varphi_1\end{array}\\[1ex]
\begin{array}{c}\chi_1\\ \chi_2\\ \chi_3 \\ \chi_4\\ \chi_{5,j}\end{array} &
\left[\begin{array}{cc}1&0\\1&0\\1&1\\1&1\\2&1\end{array}\right]
&\begin{array}{c}\\ \\ \\ \\1\le j\le 2^{d-2}-1\end{array}
\end{array}$$
$$\left(\mbox{resp.}\quad
\begin{array}{ccc}
&\begin{array}{c@{}c}\varphi_0\,&\,\varphi_1\end{array}\\[1ex]
\begin{array}{c}\chi_1\\ \chi_2\\ \chi_3 \\ \chi_4\\ \chi_{5,j}\end{array} &
\left[\begin{array}{cc}1&0\\1&0\\1&1\\1&1\\0&1\end{array}\right]
&\begin{array}{c}\\ \\ \\ \\1\le j\le 2^{d-2}-1\end{array}
\end{array}
\right).$$
\vspace{2ex}
\hrule
\end{figure}

Let $i\in\{1,2\}$ and $c\in\{0,1\}$ such that $B$ is Morita equivalent to $\Lambda_{i,c}$.
Let $S_0$ (resp. $S_1$) be a simple $\Lambda_{i,c}$-module corresponding to the
vertex $0$ (resp. $1$) of $Q_i$. Let $T_0$ (resp. $T_1$) be a simple $B$-module corresponding to
$S_0$ (resp. $S_1$) under the Morita equivalence between $B$ and $\Lambda_{i,c}$.
The radical series of the projective indecomposable $\Lambda_{i,c}$-modules (and hence of the projective indecomposable $B$-modules) can be described by the pictures in Figure \ref{fig:projs},
where we use $0$ (resp. $1$) as shorthand for $S_0$ (resp. $S_1$). 
The radical series length of the projective 
indecomposable $\Lambda_{1,c}$-modules $P_0$ and $P_1$ is $3\cdot 2^{d-2}+1$.
The radical series
length of the projective indecomposable $\Lambda_{2,c}$-module  $P_0$ (resp. $P_1$) 
is $4$ (resp. $4$ if $d=3$ and $2^{d-2}+1$ if $d\ge 4$). 
\begin{figure}[ht] \hrule \caption{\label{fig:projs} The projective indecomposable
$\Lambda_{i,c}$-modules $P_0$ and $P_1$ for $i=1$ (resp. $i=2$).}
$$P_0=\vcenter{ \xymatrix @R=.1pc @C=.5pc{&0&\\
0\ar@{.}[rddddddddd]&&1\\1&&0\\0&&0\\ 0&&1\\:&&:\\:&&:\\1&&0\\0&&0\\
0&&1\\1&&0\\&0&}},\qquad 
P_1=\vcenter{ \xymatrix @R=.1pc @C=.3pc{1\\0\\0\\ :\\:\\1\\0\\0\\1\\0\\0\\1}}$$
$$\left(\mbox{resp.}\qquad P_0=\vcenter{ \xymatrix @R=.1pc @C=.1pc{&0&\\
0\ar@{.}[rdd]&&1\\1&&0\\&0&}},\qquad 
P_1=\vcenter{ \xymatrix @R=.1pc @C=.1pc{&1&\\ 1&&0\\1&&0\\
:&&\\:&&\\1&&\\&1&}}
\right).$$
\vspace{2ex}
\hrule
\end{figure}


\subsection{Ordinary characters for dihedral blocks with two simple modules}
\label{ss:ordinary}

We now describe the ordinary characters belonging to $B$ as given in \cite{brauer2}.
Since $B$ contains exactly two isomorphism classes of simple $kG$-modules, this means that in 
the notation of \cite[Sect. 4]{brauer2} we are in Case $(ab)$ (see \cite[Thm. 2]{brauer2}).

For $2\le\ell\le d-1$, let  $\zeta_{2^\ell}$ be a fixed primitive 
$2^\ell$-th root of unity in an algebraic closure of $F$. Let
$$\chi_1,\chi_2,\chi_3,\chi_4,\qquad \chi_{5,j}, 1\le j\le 2^{d-2}-1,$$
be the ordinary irreducible characters of $G$ belonging to $B$, corresponding to the rows of the
decomposition matrices in Figure \ref{fig:decomp}. Let $\delta$ be an 
element of order $2^{d-1}$ in $D$. By \cite{brauer2}, there is a block $b_\delta$ of 
$kC_G(\delta)$ with $b_\delta^G=B$ which contains a unique $2$-modular character 
$\varphi^{(\delta)}$ such that the following is true. There is an ordering of 
$(1,2,\ldots,2^{d-2}-1)$ such that for $1\le i\le 2^{d-2}-1$ and $r$ odd,
\begin{equation}
\label{eq:great1}
\chi_{5,j}(\delta^r)=(\zeta_{2^{d-1}}^{rj}+\zeta_{2^{d-1}}^{-rj})\cdot \varphi^{(\delta)}(1).
\end{equation}

Note that $W$ contains all roots of unity of order not divisible by $2$. Hence by 
\cite{brauer2} and by \cite{fong}, the characters 
$\chi_1,\chi_2,\chi_3,\chi_4$
correspond to simple $FG$-modules. On the other hand, the characters 
$\chi_{5,j}$, $j=1,\ldots,2^{d-2}-1$,
fall into $d-2$ Galois orbits $\mathcal{O}_2,\ldots, \mathcal{O}_{d-1}$ under the action of $\mathrm{Gal}(F(\zeta_{2^{d-1}}+\zeta_{2^{d-1}}^{-1})/F)$. Namely for $2\le\ell\le d-1$, $\mathcal{O}_{\ell}=\{ \chi_{5,2^{d-1-\ell}(2u-1)} \;|\; 1\le u\le 2^{\ell-2}\}$. The field generated by the character values of each $\xi_\ell\in\mathcal{O}_\ell$ over $F$ is $F(\zeta_{2^\ell}+\zeta_{2^\ell}^{-1})$. Hence by \cite{fong}, each $\xi_\ell$ corresponds to an absolutely irreducible $F(\zeta_{2^\ell}+\zeta_{2^\ell}^{-1})G$-module $X_\ell$. 
By \cite[Satz V.14.9]{hup},  this implies that for $2\le\ell\le d-1$, the Schur index of each $\xi_\ell\in\mathcal{O}_\ell$ over $F$ is $1$. Hence we obtain $d-2$ non-isomorphic simple $FG$-modules $V_2,\ldots,V_{d-1}$
with characters $\rho_2,\ldots, \rho_{d-1}$ satisfying
\begin{equation}
\label{eq:goodchar1}
\rho_\ell =\sum_{\xi_\ell\in\mathcal{O}_\ell}\xi_\ell = 
\sum_{u=1}^{2^{\ell-2}} \chi_{5,2^{d-1-\ell}(2u-1)} 
\qquad\mbox{for $2\le \ell \le d-1$.}
\end{equation}
By \cite[Hilfssatz V.14.7]{hup}, $\mathrm{End}_{FG}(V_\ell)$ is a commutative $F$-algebra isomorphic to the field generated over $F$ by the character values of any $\xi_\ell\in\mathcal{O}_\ell$. This means
\begin{equation}
\label{eq:goodendos}
\mathrm{End}_{FG}(V_\ell)\cong F(\zeta_{2^\ell}+\zeta_{2^\ell}^{-1})\qquad\mbox{for $2\le \ell \le d-1$.}
\end{equation}

By \cite{brauer2}, the characters $\chi_{5,j}$ have the same degree $x$ for 
$1\le j\le 2^{d-2}-1$. The characters $\chi_1,\chi_2,\chi_3,\chi_4$ have height $0$ and 
$\chi_{5,j}$, $1\le j\le 2^{d-2}-1$, have height $1$. 
Let $C$ be the conjugacy class in $G$ of $\delta$, and let $t(C)\in WG$ be the 
class sum of $C$. Using the same arguments as in \cite[Sect. 3.4]{3sim},
we obtain the following action of $t(C)$ on $V_\ell$ for $2\le \ell\le d-1$. 
There exists a unit $w$ in $W$ such that for 
$2\le \ell \le d-1$, the action of $t(C)$ on $V_\ell$ is given as multiplication by
\begin{equation}
\label{eq:thatsit}
w\cdot (\zeta_{2^{d-1}}^{2^{d-1-\ell}}+\zeta_{2^{d-1}}^{-2^{d-1-\ell}})
\end{equation}
when we identify $\mathrm{End}_{FG}(V_\ell)$ with $\mathrm{End}_{F(\zeta_{2^\ell}+\zeta_{2^\ell}^{-1})G}(X_\ell)$ for an absolutely irreducible $F(\zeta_{2^\ell}+\zeta_{2^\ell}^{-1})G$-constituent $X_\ell$ of $V_\ell$ with character $\chi_{5,2^{d-1-\ell}}$. 

\begin{dfn}
\label{def:seemtoneed}
Define
$$p_d(t)=\prod_{\ell=2}^{d-1} \mathrm{min.pol.}_F(\zeta_{2^\ell}+\zeta_{2^\ell}^{-1}),$$
where
\begin{eqnarray*}
\mathrm{min.pol.}_F(\zeta_{2^2}+\zeta_{2^2}^{-1})(t)&=&t,\\ \mathrm{min.pol.}_F(\zeta_{2^{\ell}}+\zeta_{2^{\ell}}^{-1})(t)
&=&\left(\mathrm{min.pol.}_F(\zeta_{2^{\ell-1}}+\zeta_{2^{\ell-1}}^{-1})(t)\right)^2-2 \qquad \mbox{for $\ell\ge 3$}.
\end{eqnarray*}

Define $R'$ to be the $W$-algebra $R'=W[[t]]/(p_d(t))$. Then $R'$ is a complete local 
commutative Noetherian ring with residue field $k$. 
\end{dfn}


\section{Principal dihedral blocks with two simple modules}
\label{s:principal}

In this section, we prove Theorem \ref{thm:extraout}. 
Let $k$, $d$, $G$, $B$ and $D$ be as in Hypothesis \ref{hyp:throughout}.
For $c\in\{0,1\}$, let $\Lambda_{1,c}$ and $\Lambda_{2,c}$
be the basic $k$-algebras introduced in Section  \ref{ss:basicalgebras} (see Figure \ref{fig:basic}).

It follows from  \cite[Cor. 4]{parameter} that for each $d\ge 3$ and each $i\in\{1,2\}$, there 
exists a group with dihedral Sylow $2$-subgroups of order $2^d$
whose principal $2$-modular block is Morita equivalent to 
$\Lambda_{i,0}$.
Using the results \cite[Thm. 2 and Cor. 3]{parameter}, we now prove that if $B$ is a principal block, 
then the parameter $c$ must always be equal to zero. Theorem \ref{thm:extraout} is then an
immediate consequence, since for all $i,i'\in\{1,2\}$ and 
$c,c'\in\{0,1\}$, $\Lambda_{i,c}$ is isomorphic to $\Lambda_{i',c'}$ if  and only $i=i'$ and $c=c'$ 
(see \cite[Sect. VI.8]{erd}).

\begin{thm}
\label{thm:principalblocks}
Assume Hypothesis $\ref{hyp:throughout}$. Further assume that $B$ is the principal block of $kG$.
\begin{itemize}
\item[(i)] There exists $i\in\{1,2\}$ such that $B$ is Morita equivalent to $\Lambda_{i,0}$, i.e.
the parameter $c$ is zero. 
\item[(ii)] There exists an involution $\tau$ in $G$ such that if $X$ is a uniserial $B$-module with 
composition factors $k,k$, then $\mathrm{Res}^G_{\langle\tau\rangle}(X) \cong k\langle\tau\rangle$.
\end{itemize}
\end{thm}

\begin{proof}
Let $O_{2'}(G)$ be the maximal normal subgroup of $G$ of odd order, and let 
$\overline{G}=G/O_{2'}(G)$.
Since $k$ has characteristic $2$, the blocks of $k\overline{G}$
correspond to the blocks of $kG$ whose primitive central
idempotents occur in the decomposition of the central idempotent $\frac{1}{\# O_{2'}(G)}
\sum_{g\in O_{2'}(G)}g$. In particular, the principal block $B$ of $kG$ is isomorphic to the
principal block $\overline{B}$ of $k\overline{G}$. If $V,V'$ are $kG$-modules belonging to
$B$ then $O_{2'}(G)$ acts trivially on $V,V'$ and $V,V'$ are inflations of $k\overline{G}$-modules
belonging to $\overline{B}$. 
Since $O_{2'}(G)$ has odd order, it also follows that each element $\overline{t}\in\overline{G}$ 
of order $2$ has a preimage $t\in G$ of order $2$. Thus it suffices to prove Theorem
\ref{thm:principalblocks} for $G=\overline{G}$ and $B=\overline{B}$.

By the classification of the finite groups with dihedral Sylow $2$-subgroups
in \cite{gowa}, it follows that $\overline{G}=G/O_{2'}(G)$ is isomorphic to either
\begin{itemize}
\item[(a)] a subgroup of $\mathrm{P\Gamma L}_2(\mathbb{F}_q)$, for some odd prime power $q$, 
that contains $\mathrm{PSL}_2(\mathbb{F}_q)$, or
\item[(b)] the alternating group $A_7$, or
\item[(c)] a Sylow $2$-subgroup of $G$.
\end{itemize}
By assumption, $B$ has precisely two isomorphism classes of
simple modules. Since in case (b) (resp. case (c)), the principal $2$-modular block 
has  precisely 3 (resp. 1) isomorphism classes of simple modules, we are in case (a). 

Let $q=p^f$ be an odd prime power. Then
$\mathrm{P\Gamma L}_2(\mathbb{F}_q)$ is the group of semilinear fractional maps
\begin{equation}
\label{eq:elements}
\sigma_{a,b,c,d,\nu}: x\mapsto \frac{ax^{\nu}+b}{cx^{\nu}+d}
\end{equation}
over $\mathbb{F}_q$, where $a,b,c,d\in \mathbb{F}_q$ with $ad-bc\neq 0$ and
$\nu$ ranges over all automorphisms in 
$\mathrm{Gal}(\mathbb{F}_q/\mathbb{F}_p)\cong \mathbb{Z}/f$. Let $\varphi:\mathbb{F}_q\to
\mathbb{F}_q$ be the Frobenius automorphism, i.e. $\varphi(x)=x^p$ for all $x\in \mathbb{F}_q$.
Then $\mathrm{Gal}(\mathbb{F}_q/\mathbb{F}_p)=\langle\varphi\rangle$.
The projective general linear group $\mathrm{PGL}_2(\mathbb{F}_q)$ is 
the normal subgroup of $\mathrm{P\Gamma L}_2(\mathbb{F}_q)$
consisting of all $\sigma_{a,b,c,d,\nu}
\in \mathrm{P\Gamma L}_2(\mathbb{F}_q)$ for which $\nu$ is the identity
$\mathrm{id}_{\mathbb{F}_q}$.
Recall that $\mathrm{P\Gamma L}_2(\mathbb{F}_q)$ is the automorphism group of
$\mathrm{PGL}_2(\mathbb{F}_q)$. We have a short exact sequence of finite groups
\begin{equation}
\label{eq:sespgamma1}
1\to \mathrm{PGL}_2(\mathbb{F}_q) \to \mathrm{P\Gamma L}_2(\mathbb{F}_q)
\to \mathrm{Gal}(\mathbb{F}_q/\mathbb{F}_p)\to 1.
\end{equation}
Sequence $(\ref{eq:sespgamma1})$ splits on the right, since the group 
\begin{equation}
\label{eq:A}
A=\{\sigma_\nu\;|\; \nu\in \mathrm{Gal}(\mathbb{F}_q/\mathbb{F}_p)\}
\qquad\mbox{ where $\sigma_\nu:x\mapsto x^\nu$}
\end{equation}
is a subgroup of $\mathrm{P\Gamma L}_2(\mathbb{F}_q)$ which has trivial intersection with
$\mathrm{PGL}_2(\mathbb{F}_q)$.
Since $\mathrm{PSL}_2(\mathbb{F}_q)$ is a characteristic subgroup of 
$\mathrm{PGL}_2(\mathbb{F}_q)$, we obtain a short exact sequence
\begin{equation}
\label{eq:sespgamma2}
1\to \mathrm{PSL}_2(\mathbb{F}_q) \to \mathrm{P\Gamma L}_2(\mathbb{F}_q)
\xrightarrow{\pi} Q\to 1
\end{equation}
where $Q$ lies in a short exact sequence
\begin{equation}
\label{eq:anotherses}
1\to \mathrm{PGL}_2(\mathbb{F}_q)/\mathrm{PSL}_2(\mathbb{F}_q)\to 
Q \to \mathrm{Gal}(\mathbb{F}_q/\mathbb{F}_p)\to 1.
\end{equation}
Since $\mathrm{PGL}_2(\mathbb{F}_q)/\mathrm{PSL}_2(\mathbb{F}_q)$ has order $2$,
$Q$ is abelian; in fact, $Q\cong\mathbb{Z}/2\times \mathbb{Z}/f$.

Let $H$ be a subgroup of $\mathrm{P\Gamma L}_2(\mathbb{F}_q)$ such that $H$ contains 
$\mathrm{PSL}_2(\mathbb{F}_q)$ and has dihedral Sylow $2$-subgroups. Then $H$ 
lies in a short exact sequence
\begin{equation}
\label{eq:sesG}
1\to \mathrm{PSL}_2(\mathbb{F}_q) \to H \xrightarrow{\pi|_G} C\to 1
\end{equation}
where $C$ is a subgroup of $Q$. Let $B_0(H)$ be the principal block of $kH$.
We first analyze the structure of $H$ and $B_0(H)$ in several steps.

\medskip

\noindent\textit{Claim $1$.} The order of $C$ is odd or exactly divisible by $2$.

\medskip

\noindent\textit{Proof of Claim $1$.} 
If $P$ is a Sylow $2$-subgroup of  $\mathrm{PSL}_2(\mathbb{F}_q)$, then there exists a Sylow 
$2$-subgroup $\tilde{P}$ of $H$ that contains $P$. Let $2^s$ be the maximal $2$-power dividing
$\#C$, so that $[\tilde{P}:P]=2^s$. Since $P$ is a normal subgroup of $\tilde{P}$ of order at least $4$, 
it follows from \cite[Prop. (1B)]{brauer2} that $[\tilde{P}:P]=1$ or $2$, which proves Claim $1$.

\medskip

\noindent\textit{Claim $2$.} If the order of $C$ is even, then $H$ contains 
$\mathrm{PGL}_2(\mathbb{F}_q)$ with odd index. 

\medskip

\noindent\textit{Proof of Claim $2$.} 
If $f$ is odd, Claim $2$ follows since then $Q\cong \mathbb{Z}/(2f)\,$ has a unique element of order $2$.
Assume now that $f$ is even, so $f=2f'$ for some $f'\ge 1$ and $q\equiv 1\mod 4$. Let
$\omega\in \mathbb{F}_q^*$ be an element of maximal $2$-power order. Then 
$Q\cong \mathbb{Z}/2\times \mathbb{Z}/f$ has exactly 3 elements of order $2$
given by the images in $Q$ of the following elements in $\mathrm{P\Gamma L}_2(\mathbb{F}_q)$:
\begin{equation}
\label{eq:2elements}
\lambda: x\mapsto -\omega/x,\qquad \psi:x\mapsto x^{\varphi^{f'}},\qquad 
\lambda\circ\psi:x\mapsto -\omega/x^{\varphi^{f'}}.
\end{equation}
Let $\mu\in\{\lambda,\psi,\lambda\circ\psi\}$. Analyzing the Sylow $2$-subgroups of
$\langle \mathrm{PSL}_2(\mathbb{F}_q),\mu\rangle$, we see that they are dihedral if and only if 
$\mu=\lambda$.
Since $\langle \mathrm{PSL}_2(\mathbb{F}_q),\lambda \rangle=\mathrm{PGL}_2(\mathbb{F}_q)$,
this implies Claim $2$.

\medskip

\noindent\textit{Claim $3$.} The order of $C$ is even (resp. odd) if and only if $B_0(H)$ has precisely
$2$ (resp. $3$) isomorphism classes of simple modules. 

\medskip

\noindent\textit{Proof of Claim $3$.}  
We use Clifford Theory to prove Claim $3$.
If $C$ is even (resp. odd), let $N=\mathrm{PGL}_2(\mathbb{F}_q)$ (resp. 
$N=\mathrm{PSL}_2(\mathbb{F}_q)$). By Claim $2$, $N$ is a normal subgroup of $H$ with odd index.
It follows from the description of the principal block $B_0(N)$ of $kN$ as given in \cite[p. 294--296]{erd}
that there are precisely 2 (resp. 3) isomorphism classes of simple $B_0(N)$-modules and that
every simple $B_0(N)$-module occurs as composition factor of every projective indecomposable
$B_0(N)$-module. In both cases,
it follows from \cite[Sect. 15]{alp} that the principal block $B_0(H)$ of $kH$ only covers the principal
block $B_0(N)$ of $kN$. Moreover, if $M$ is a $B_0(H)$-module, then $\mathrm{Res}^H_N(M)$ 
belongs to $B_0(N)$. 

Because $B_0(H)$ and $B_0(N)$ both have at most three isomorphism classes of simple modules
and because $[H:N]$ is odd, it follows 
from \cite[Hauptsatz V.17.3]{hup} that for any simple $B_0(H)$-module $E$,
there exists a simple $B_0(N)$-module $S$ and a positive integer $e$ such that
$\mathrm{Res}^H_N(E)=S^e$.
Since every simple $B_0(N)$-module occurs as composition factor of every projective indecomposable
$B_0(N)$-module, this implies
that $B_0(N)$ cannot have more isomorphism classes of simple modules than $B_0(H)$.
In particular, $B_0(H)$ has at least two isomorphism classes of simple modules. Moreover,
if $B_0(H)$ has precisely two isomorphism classes of simple modules, then $\#C$
is even.

Conversely, suppose that $\#C$ is even, so that $N=\mathrm{PGL}_2(\mathbb{F}_q)$.
Let $E$ be a non-trivial simple $B_0(H)$-module. Then
$\mathrm{Res}^H_N(E)\cong T^e$, where $T$ is the unique (up to isomorphism) non-trivial
simple $B_0(N)$-module and $e\in\mathbb{Z}^+$. 
Let $N_1=\mathrm{PSL}_2(\mathbb{F}_q)$, and let $\tau\in N-N_1$.
Then $\mathrm{Res}^N_{N_1}(T)\cong V \oplus V'$, where $V$ and $V'$
are representatives of the two isomorphism classes of non-trivial simple $B_0(N_1)$-modules
and $V'$ is isomorphic to the conjugate ${}^\tau V$.
Hence $\mathrm{Res}^H_{N_1}(E)\cong (V\oplus {}^\tau V)^e$. By \cite[Hauptsatz V.17.3]{hup},
the subgroup $Z$ of $H$ of all $h\in H$ with ${}^h V \cong V$ 
has index 2 in $H$
and there exists a simple $kZ$-submodule $W$ of $\mathrm{Res}^H_Z(E)$ with 
$E\cong \mathrm{Ind}_Z^H (W)$ and $\mathrm{Res}^H_{N_1}(W) \cong V^e$. Moreover, 
$\mathrm{Res}^H_Z(E) = W \oplus \tau W$, $E\cong \mathrm{Ind}_Z^H (\tau W)$
and $\mathrm{Res}^H_{N_1}(\tau W) \cong ({}^\tau V)^e$. Since $B_0(Z)$ has at most three
isomorphism classes of simple modules, this implies that $B_0(Z)$ has precisely
three isomorphism classes of simple modules represented by $k$, $W$ and $\tau W$.
By \cite[Hauptsatz V.17.3]{hup}, it follows that every non-trivial simple $B_0(H)$-module is isomorphic to
$\mathrm{Ind}_Z^H (W)\cong \mathrm{Ind}_Z^H (\tau W)$. Therefore, $B_0(H)$ has precisely two 
isomorphism classes of simple modules, which implies Claim $3$.

\medskip

We next consider the case when $H=\overline{G}=G/O_{2'}(G)$. In particular, 
$B_0(H)=\overline{B}$.
Since by assumption there are precisely $2$ isomorphism classes of simple $\overline{B}$-modules,
it follows by Claims 2 and 3 that  $\mathrm{PGL}_2(\mathbb{F}_q)$ is a normal subgroup of
$\overline{G}$ such that the quotient group $\overline{G}/\mathrm{PGL}_2(\mathbb{F}_q)$ is a 
subgroup of odd order of the cyclic group $\mathrm{Gal}(\mathbb{F}_q/\mathbb{F}_p)=
\langle \varphi\rangle$ of order $f$. Using the splitting of the short exact sequence 
$(\ref{eq:sespgamma1})$, it follows that there exists a divisor $\ell$ of $f$ such that 
\begin{equation}
\label{eq:presentG}
\overline{G}=\langle \mathrm{PGL}_2(\mathbb{F}_q),\sigma_{\varphi^\ell}\rangle
\end{equation}
where $\sigma_{\varphi^\ell}$ is as defined in $(\ref{eq:A})$. Since $\#
(\overline{G}/\mathrm{PGL}_2(\mathbb{F}_q))=\#\langle \varphi^\ell\rangle
=f/\ell$ is odd, it follows that the maximal power of $2$ that divides $f$ also divides $\ell$.

\medskip

\noindent\textit{Claim $4$.} There exists a group $\hat{G}$ with generalized quaternion Sylow 
$2$-subgroups such that $\hat{G}$ has no non-trivial normal subgroups of odd order and 
such that $\hat{G}/Z(\hat{G})\cong \overline{G}$, where $Z(\hat{G})$ denotes the center of $\hat{G}$.

\medskip

\noindent\textit{Proof of Claim $4$.} 
We use the presentation of $\overline{G}$ in $(\ref{eq:presentG})$.
To construct $\hat{G}$, we make use of the fact that $\mathrm{SL}_2(\mathbb{F}_{q^2})$ has
generalized quaternion Sylow $2$-subgroups (see for example \cite[Satz II.8.10]{hup}).
Let $\omega\in\mathbb{F}_q^*$ be of maximal $2$-power order, and
let $\omega_0\in\mathbb{F}_{q^2}$ be such that $\omega_0^2=\omega$. Define 
$$z=\left(\begin{array}{cc}\omega_0&0\\0&\omega_0^{-1}\end{array}\right)\in 
\mathrm{SL}_2(\mathbb{F}_{q^2}).$$
View $\mathrm{SL}_2(\mathbb{F}_q)$ as a subgroup of $\mathrm{SL}_2(\mathbb{F}_{q^2})$ and
define
$$\hat{G}_1=\langle \mathrm{SL}_2(\mathbb{F}_q),z\rangle \le \mathrm{SL}_2(\mathbb{F}_{q^2}).$$
Let $\hat{\phi}$ be the automorphism of $\mathrm{SL}_2(\mathbb{F}_{q^2})$ which raises
all entries of a matrix in $\mathrm{SL}_2(\mathbb{F}_{q^2})$ to their $p^{\mathrm{th}}$ powers.
Then $\hat{\phi}$ has order $2f$ and $\hat{\phi}^{f+\ell}$ has  order $f/\ell$, which is the order
of $\sigma_{\varphi^\ell}$. Since $\hat{\phi}^{f+\ell}(z)=z^{p^{f+\ell}}=-z^{p^\ell}\in \hat{G}_1$,
$\hat{\phi}^{f+\ell}$ defines an automorphism of $\hat{G}_1$ of order $f/\ell$.
Define $\hat{G}$ to be the corresponding semidirect product 
$$\hat{G}=\hat{G}_1 \semid \langle \hat{\phi}^{f+\ell}\rangle.$$
The center $Z(\hat{G})$ obviously contains $\{\pm I\}<\hat{G}_1$,
where $I$ is the $2\times 2$ identity matrix. Since $\hat{G}_1/\{\pm I\}$ can be identified with the 
subgroup of $\mathrm{PSL}_2(\mathbb{F}_{q^2})$ which is generated by 
$\mathrm{PSL}_2(\mathbb{F}_{q})$ together with the element 
$\overline{z}:x\mapsto \frac{\omega_0\,x}{\omega_0^{-1}}=\omega\,x$
which lies in $\mathrm{PGL}_2(\mathbb{F}_q)-\mathrm{PSL}_2(\mathbb{F}_q)$, 
it follows that $\hat{G}_1/\{\pm I\}$ can be identified with 
$\mathrm{PGL}_2(\mathbb{F}_q)$. Under this identification, $\hat{\phi}^{f+\ell}$ sends each element
$\sigma_{a,b,c,d,\mathrm{id}_{\mathbb{F}_q}}\in \mathrm{PGL}_2(\mathbb{F}_q)=\hat{G}_1/\{\pm I\}$ to 
$\sigma_{\varphi^\ell}\circ \sigma_{a,b,c,d,\mathrm{id}_{\mathbb{F}_q}}\circ (\sigma_{\varphi^\ell})^{-1}$. Thus it follows that $\hat{G}/\{\pm I\}$ is isomorphic
to the group $\langle \mathrm{PGL}_2(\mathbb{F}_q),\sigma_{\varphi^\ell}\rangle=\overline{G}$. 
Because $O_{2'}(\overline{G})=1$, $\hat{G}$ has no non-trivial normal subgroups of odd order. 
Since the center of $\mathrm{PGL}_2(\mathbb{F}_q)$ is trivial, it follows that the center of
$\overline{G}$ is also trivial, which implies $Z(\hat{G})=\{\pm I\}$.
Comparing orders, we see that 
the Sylow $2$-subgroups of $\hat{G}$ are isomorphic to Sylow $2$-subgroups of
$\mathrm{SL}_2(\mathbb{F}_{q^2})$, which
proves Claim 4.

\medskip

We now prove parts (i) and (ii) of Theorem \ref{thm:principalblocks}. 
As discussed in the first paragraph of the proof, it suffices to prove Theorem
\ref{thm:principalblocks} for $G=\overline{G}$ and $B=\overline{B}$.
Part (i) of Theorem \ref{thm:principalblocks} follows thus from
\cite[Cor. 3]{parameter} together with Claim 4.
To prove part (ii) of Theorem \ref{thm:principalblocks}, we define an element $\overline{\tau}\in
\overline{G}$ as follows.
If $q\equiv 1\mod 4$, we let $\omega\in\mathbb{F}_q^*$ be of maximal $2$-power order and define
$\overline{\tau}:x\mapsto -\omega/x$. If $q\equiv 3\mod 4$, then $-1$ is not a square and we define
$\overline{\tau}:x\mapsto -x$. In both cases, $\overline{\tau}$ has order $2$ and lies in  
$\mathrm{PGL}_2(\mathbb{F}_q)-\mathrm{PSL}_2(\mathbb{F}_q)$. Let 
$U=\langle \mathrm{PSL}_2(\mathbb{F}_q),\sigma_{\varphi^\ell}\rangle$, so 
$\overline{G}=\langle U,\overline{\tau}\rangle$.
Let $X$ be a uniserial $\overline{B}$-module with composition factors $k,k$. By Claim 3, the principal
block $B_0(U)$ has exactly 3 isomorphism classes of simple modules. Thus it follows from the description of $B_0(U)$ as given in \cite[p. 295--296]{erd} that $\mathrm{Ext}^1_{kU}(k,k)=0$.
Hence $\mathrm{Res}^{\overline{G}}_U(X)$ is a trivial $kU$-module. Since the action of 
$\overline{G}$ on $X$ is not trivial, it follows that 
$\mathrm{Res}^{\overline{G}}_{\langle \overline{\tau}\rangle}(X)$ is  a non-trivial 
$k\langle\overline{\tau}\rangle$-module of $k$-dimension $2$. Hence
$\mathrm{Res}^G_{\langle \overline{\tau}\rangle}(X)\cong k\langle \overline{\tau} \rangle$,
which implies part (ii) of Theorem \ref{thm:principalblocks}.
\end{proof}

\begin{rem}
\label{rem:pgammal}
The proof of Theorem \ref{thm:principalblocks} shows the following.
Suppose $H$ is a subgroup of
$\mathrm{P\Gamma L}_2(\mathbb{F}_q)$, where $q$ is an odd prime power, 
such that $H$ contains $\mathrm{PSL}_2(\mathbb{F}_q)$ and $H$ has dihedral Sylow
$2$-subgroups. Then either $H/\mathrm{PSL}_2(\mathbb{F}_q)$ is cyclic of odd order
in which case the principal block of $kH$ has exactly 3 isomorphism classes of simple modules, or $H$
contains $\mathrm{PGL}_2(\mathbb{F}_q)$ and $H/\mathrm{PGL}_2(\mathbb{F}_q)$ is cyclic
of odd order in which case the principal block of $kH$ has exactly 2 isomorphism classes of simple 
modules.
\end{rem}


\section{Stable endomorphism rings}
\label{s:stablend}

In this section, we consider the basic $k$-algebras $\Lambda_{1,c}$ and 
$\Lambda_{2,c}$ from Section  \ref{ss:basicalgebras} (see Figure \ref{fig:basic})
and determine all indecomposable modules for these algebras whose stable endomorphism 
rings are isomorphic to $k$. As in Hypothesis \ref{hyp:throughout}, we assume $d\ge 3$.
In particular, this will determine all indecomposable $B$-modules whose stable endomorphism rings
are isomorphic to $k$ for $B$ as in Hypothesis \ref{hyp:throughout}.

Let $i\in\{1,2\}$, let $c\in\{0,1\}$ and let $d\ge 3$.
It follows from the definition of string algebras in \cite[Sect. 3]{buri} that 
$\Lambda_{i,c}/\mathrm{soc}(\Lambda_{i,c})$ is a string algebra.
Therefore, one can see as in \cite[Sect. I.8.11]{erd} that
the isomorphism classes of all non-projective indecomposable 
$\Lambda_{i,c}$-modules are given by string and band modules as defined in \cite[Sect. 3]{buri}.
We will give a brief introduction into the representation theory of $\Lambda_{i,c}$ in 
Section \ref{s:stringband}. In what follows, we freely use the notation from 
Section \ref{s:stringband} without always referring to specific results.

Before we can describe all indecomposable $\Lambda_{i,c}$-modules whose stable endomorphism
rings are isomorphic to $k$, we need to define some particular $\Lambda_{i,c}$-modules.
As before, $\Omega$ denotes the syzygy functor.

\begin{dfn}
\label{def:modules}
Let $i\in\{1,2\}$, let $c\in\{0,1\}$ and let $d\ge 3$.
\begin{enumerate}
\item[(i)] Define the following uniserial $\Lambda_{i,c}$-modules which are
uniquely determined by their descending composition factors:
$$S_{01}=\begin{array}{c}0\\1\end{array},\;S_{10}=\begin{array}{c}1\\0\end{array},\;
S_{001}=\begin{array}{c}0\\0\\1\end{array},\; S_{100}=\begin{array}{c}1\\0\\0\end{array}.$$
In other words, these are the string modules
$S_{01}=M(\beta)$, $S_{10}=M(\gamma)$, $S_{001}=M(\beta\alpha)$, $S_{100}=
M(\alpha\gamma)$.
\item[(ii)] Let $C_{010}=\beta^{-1}\gamma^{-1}(\alpha^{-1}\beta^{-1}\gamma^{-1}))^{2^{d-2}-1}$
if $i=1$ $($resp. $C_{010}=\beta^{-1}\gamma^{-1}$ 
if $i=2$$)$. Define the following band and string modules for $\Lambda_{i,c}$:
$$S^{(\lambda)}_{010}=M(\alpha \, C_{010},\lambda,1)\quad\mbox{ for $\lambda\in k^*$}
\quad \mbox{ and } \quad S^{(0)}_{010}=M(C_{010}).$$
Note that for $i\in\{1,2\}$, $\Omega(S^{(\lambda)}_{010})\cong S^{(\lambda)}_{010}$ if 
$c=0$ and
$\Omega(S^{(\lambda)}_{010})\cong S^{(1-\lambda)}_{010}$ if $c=1$.
\end{enumerate}
\end{dfn}

\begin{prop}
\label{prop:stablend}
Let $i\in\{1,2\}$, let $c\in\{0,1\}$ and let $d\ge 3$.
Suppose $\mathfrak{C}$ is a component of the stable Auslander-Reiten 
quiver of $\Lambda_{i,c}$ containing a module whose endomorphism ring is
isomorphic to $k$. Then 
$\mathfrak{C}$ contains 
$S_0$, $\Omega(S_0)$, $S_1$ or $S_{001}$.
\begin{enumerate}
\item[(i)] Suppose $\mathfrak{C}$ contains $S_0$. Then $\mathfrak{C}$ and 
$\Omega(\mathfrak{C})$ are both of type $\mathbb{Z}A_\infty^\infty$. Each module
$M$ in $\mathfrak{C}\cup\Omega(\mathfrak{C})$ satisfies
$\underline{\mathrm{End}}_{\Lambda_{i,c}}(M)\cong k$ and 
$\mathrm{Ext}^1_{\Lambda_{i,c}}(M,M)\cong k$.
\item[(ii)] Let $X=S_{001}$ if $i=1$ $($resp. let $X=S_1$ if $i=2$$)$ and
suppose $\mathfrak{C}$ contains $X$. Then $\mathfrak{C}$ and 
$\Omega(\mathfrak{C})$ are both of type $\mathbb{Z}A_\infty^\infty$, and 
$\mathfrak{C}=\Omega(\mathfrak{C})$ exactly when $d=3$. 
If $M$ lies in $\mathfrak{C}\cup\Omega(\mathfrak{C})$ with
$\underline{\mathrm{End}}_{\Lambda_{i,c}}(M)\cong k$, then 
$M$ is isomorphic to $\Omega^j(X)$ for some integer $j$
and $\mathrm{Ext}^1_{\Lambda_{i,c}}(M,M)\cong k$.
\item[(iii)] Let $Y=S_1$ if $i=1$ $($resp. let $Y=S_{001}$ if $i=2$$)$ and 
suppose $\mathfrak{C}$ contains $Y$.
Then $\mathfrak{C}$ is a $3$-tube with $Y$ belonging to its boundary, and 
$\mathfrak{C}=\Omega(\mathfrak{C})$. If $M$ lies in $\mathfrak{C}$ with
$\underline{\mathrm{End}}_{\Lambda_{i,c}}(M)\cong k$,
then $M\in\{Y,\Omega(Y),\Omega^2(Y)\}$ up to 
isomorphism, and $\mathrm{Ext}^1_{\Lambda_{i,c}}(M,M)= 0$.
\item[(iv)] Let $\lambda\in k$ and let $\mathfrak{T}^{(\lambda)}$ be the component of the 
stable Auslander-Reiten quiver of $\Lambda_{i,c}$ containing $S^{(\lambda)}_{010}$.
Then $\mathfrak{T}^{(\lambda)}$ is a one-tube with $S^{(\lambda)}_{010}$ belonging to its
boundary, and
$\underline{\mathrm{End}}_{\Lambda_{i,c}}(S^{(\lambda)}_{010})\cong k$ if and only if $c=1$. 
If $c=1$ and $M$ lies in $\mathfrak{T}^{(\lambda)}$ with
$\underline{\mathrm{End}}_{\Lambda_{i,1}}(M)\cong k$,
then $M\cong S^{(\lambda)}_{010}$ and 
$\mathrm{Ext}^1_{\Lambda_{i,1}}(M,M)\cong k$.
\end{enumerate}
If $c=0$ $($resp. $c=1$$)$, then the only components of the stable Auslander-Reiten quiver of 
$\Lambda_{i,c}$ containing modules whose stable endomorphism rings are isomorphic to $k$ 
are the ones in $(i) - (iii)$ $($resp. $(i) - (iv)$$)$.
\end{prop}

\begin{proof}
Considering the description of all non-projective indecomposable 
$\Lambda_{i,c}$-modules as string and band modules, it is straightforward to see that
for both $i=1$ and $i=2$ and all $d\ge 3$, a complete list of $\Lambda_{i,c}$-modules whose 
endomorphism rings are isomorphic to $k$ is given by the uniserial modules
$$S_0,\;S_1,\;S_{01},\;S_{10},\;S_{001},\:S_{100}.$$
Note that the uniserial modules of length $2$ lie in the  component of the stable Auslander-Reiten
quiver of $\Lambda_{i,c}$ containing $S_0$ if $i=1$ (resp. $\Omega(S_0)$ if $i=2$).
For $i\in\{1,2\}$, $S_{100}=\Omega^{-2}(S_{001})$.

Let first $i=2$ and $c=0$. If $d=3$, then $\Lambda_{2,0}$ is Morita equivalent to $kS_4$.
Therefore, Proposition \ref{prop:stablend} follows from \cite[Thm. 1.1]{blello}. 
Let now $d>3$. Going through
the arguments for $kS_4$ in \cite[Sect. 7]{blello}, it is not hard to see that we can 
generalize them to $d>3$. The only difference is that for $d>3$, strings and bands can contain
subwords of the form $\eta,\eta^2,\eta^3,\ldots,\eta^{2^{d-2}-1}$, whereas for $d=3$ only $\eta$
can occur as such a subword. It follows that the
only components of the stable Auslander-Reiten quiver of $\Lambda_{2,0}$
that possibly contain an indecomposable module whose stable endomorphism ring is isomorphic to $k$
are either the unique $3$-tube, or the components containing a simple module $S$ or its
syzygy $\Omega(S)$. 
Similarly to the proof of \cite[Lemma 6.1]{blello},
it follows 
that the only modules in the
$3$-tube whose stable endomorphism rings are isomorphic to $k$ lie at its boundary. 
Since these modules correspond to
string modules of maximal directed strings, we see that these modules
are isomorphic to $S_{001}$, $S_{100}$ or the
uniserial module of length $2^{d-2}$ whose composition factors are all isomorphic
to $S_1$. Using the projective module $P_0$, we see that 
$\mathrm{Ext}^1_{\Lambda_{2,0}}(S_{001},S_{001})
\cong \underline{\mathrm{Hom}}_{\Lambda_{2,0}}(S_{100},S_{001}) = 0$.
This proves part (iii).
If $\mathfrak{C}$ contains $S_1$ or $\Omega(S_1)$, it follows similarly to the proof of
\cite[Lemma 5.1]{blello} that the only modules in $\mathfrak{C}$ whose stable
endomorphism rings are isomorphic to $k$ are the modules in the $\Omega$-orbit of $S_1$. 
Using the projective module $P_1$, we see that 
$\mathrm{Ext}^1_{\Lambda_{2,0}}(S_1,S_1)
\cong \mathrm{Hom}_{\Lambda_{2,0}}(\Omega(S_1),S_1)\cong k$. 
This prove part (ii).
Finally consider the components $\mathfrak{C}$ containing $S_0$ and $\Omega(S_0)$.
By \cite[Cor. 4]{parameter}, $\Lambda_{2,0}$ is Morita equivalent to the principal block $B_0$ of 
$kH$ when $H$ is the projective general linear group $\mathrm{PGL}_2(\mathbb{F}_q)$,  
where $q\equiv 3\mod 4$ and $2^d$ is the maximal $2$-power dividing $(q^2-1)$.
Hence to prove that the stable endomorphism ring of each $\Lambda_{2,0}$-module $M$ in 
$\mathfrak{C}\cup\Omega(\mathfrak{C})$ is isomorphic to $k$, it suffices to prove that
the indecomposable $B_0$-module $V$ corresponding to each such $M$ under the Morita equivalence
is an endo-trivial $kH$-module.
This follows using similar arguments as in the
proof of \cite[Lemma 4.1]{blello}. Since $\mathrm{End}_k(V)$, viewed as a $kH$-module, is stably
isomorphic to the trivial simple $kH$-module $k$, it follows that
$\mathrm{Ext}_{B_0}^1(V,V)\cong \HH^1(H,\mathrm{End}_k(V))\cong
\HH^1(H,k)\cong \mathrm{Ext}_{B_0}^1(k,k)$ is isomorphic to 
$\mathrm{Ext}^1_{\Lambda_{2,0}}(S_0,S_0)\cong k$.
This proves part (i) and completes the proof of Proposition \ref{prop:stablend} when $c=0$ and $i=2$.

Next suppose that $i=2$ and $c=1$. We can organize the string and band modules
for $\Lambda_{2,1}$ in the same manner as for $\Lambda_{2,0}$ above. The only difference
to the above case is that we need to be careful when using the projective module $P_0$ in
our arguments. Going through the arguments in \cite[Sect. 7]{blello} in this way, we see that the
only components of the stable Auslander-Reiten quiver of $\Lambda_{2,1}$ that possibly 
contain an indecomposable module whose stable endomorphism ring is isomorphic to $k$
are either the unique $3$-tube, or the components containing a simple module $S$ or its
syzygy $\Omega(S)$, or the one-tubes $\mathfrak{T}^{(\lambda)}$, $\lambda\in k$, described in
part (iv) of Proposition \ref{prop:stablend}. For the $3$-tube and the stable Auslander-Reiten 
components containing $S_1$ and $\Omega(S_1)$ we can adjust the arguments used
in the case when $i=2$ and $c=0$ to prove parts (ii) and (iii). 
Suppose now that $\lambda\in k$. If $\lambda\neq 0$, then it follows from \cite{krau} 
(see also Remark \ref{rem:bandhoms}) that
the stable endomorphism ring of each module in $\mathfrak{T}^{(\lambda)}$ that does not lie at the 
boundary (i.e. that is not isomorphic to $S_{010}^{(\lambda)}$) is not isomorphic to $k$. 
If $\lambda=0$, then the modules that do not lie at the boundary
of $\mathfrak{T}^{(0)}$ (i.e. that are not isomorphic to $S_{010}^{(0)}$)
are string modules corresponding to the strings
$\beta^{-1}\gamma^{-1}(\alpha\beta^{-1}\gamma^{-1})^n$ for $n\ge 1$. All of these have a non-zero
endomorphism whose image is isomorphic to $S^{(0)}_{010}$ and which does not factor through a 
projective module. Consider now $S^{(\lambda)}_{010}$ for $\lambda\in k$. 
The endomorphism ring of $S^{(\lambda)}_{010}$
has $k$-dimension $2$. Since $\Omega(S^{(\lambda)}_{010})\cong S^{(1-\lambda)}_{010}$, it
follows that the non-zero endomorphisms of $S^{(\lambda)}_{010}$ whose image is 
isomorphic to $S_0$
factor through $P_0$. Hence the stable endomorphism ring of $S^{(\lambda)}_{010}$ is
isomorphic to $k$. On the other hand, the non-zero homomorphisms from 
$\Omega(S^{(\lambda)}_{010})\to S^{(\lambda)}_{010}$ whose images
are isomorphic to $S_0$ do not factor through $P_0$, which implies that
$\mathrm{Ext}^1_{\Lambda_{2,1}}(S^{(\lambda)}_{010},S^{(\lambda)}_{010})\cong k$.
This proves part (iv).
Finally consider the stable Auslander-Reiten component $\mathfrak{C}_0$ containing $S_0$.
Since we cannot use the same arguments as in the case when $i=2$ and $c=0$,
we need to use the description of the homomorphisms between string modules as described
in \cite{krau} (see Remark \ref{rem:stringhoms}) to prove the statements in part (i).
Using hooks and cohooks (see Definition \ref{def:arcomps}), 
we see that all the modules in $\mathfrak{C}_0\cup\Omega(\mathfrak{C}_0)$
lie in the $\Omega$-orbits of the following string modules:
\begin{equation}
\label{eq:c0mods}
S_0,\quad M((\alpha\beta^{-1}\gamma^{-1})^n),\quad  M((\alpha^{-1}\gamma\beta)^n),
\quad n\ge 1.
\end{equation}
We have $\underline{\mathrm{End}}_{\Lambda_{2,1}}(S_0)\cong k$ and 
$\mathrm{Ext}^1_{\Lambda_{2,1}}(S_0,S_0)\cong k$. For $n\ge 1$,
let $M_n=M((\alpha\beta^{-1}\gamma^{-1})^n)$ and let $\{x_r\}_{r=0}^{3n}$ be a
canonical $k$-basis for $M_n$ relative to the representative $(\alpha\beta^{-1}\gamma^{-1})^n$ (see 
Definition \ref{def:strings}). Let $h_a=x_{3a-2}$ for $1\le a\le n$ and let $f_b=x_{3b}$ for $0\le b\le n$.
Then $h_1,\ldots, h_n$ (resp. $f_0,\ldots,f_n$) generate $M_n/\mathrm{rad}(M_n)$ 
(resp. $\mathrm{soc}(M_n)$). By Remark \ref{rem:stringhoms}, each endomorphism of $M_n$ is a 
$k$-linear combination 
of the identity morphism and the endomorphisms $\mu_{a,b}$ ($1\le a\le n,0\le b\le n$)
where $\mu_{a,b}$ sends $h_a$ to $f_b$ and all other $x_r$ to zero. Considering the projective
module $P_0$, it follows that $\mu_{a,b}+\mu_{a,b+1}+\mu_{a+1,b+1}$ factors through $P_0$
for $1\le a\le n-1,0\le b\le n-1$.  Moreover, $\mu_{1,b}$ factors through $P_0$ for $1\le b\le n$ and 
$\mu_{n,b}+\mu_{n,b+1}$ factors through $P_0$ for $0\le b\le n-1$. 
Combining these endomorphisms in a suitable way and using an inductive argument, 
we see that $\mu_{a,b}$ factors through
a direct sum of copies of $P_0$ for all $1\le a\le n,0\le b\le n$, which implies that 
$\underline{\mathrm{End}}_{\Lambda_{2,1}}(M_n)\cong k$. Since 
$\mathrm{Ext}^1_{\Lambda_{2,1}}(M_n,M_n)\cong \underline{\mathrm{Hom}}_{\Lambda_{2,1}}(
\Omega(M_n),M_n)$, we also have to consider 
$\Omega(M_n)=M(\gamma^{-1}(\alpha\beta^{-1}\gamma^{-1})^{n-1})$. 
Note that $\Omega(M_n)$ can be identified with the 
$\Lambda_{2,1}$-submodule of $M_n$ having $\{x_r\}_{r=2}^{3n}$ as a $k$-basis. 
Let $h'_1=x_2$ and $h'_a=x_{3a-2}$ for $2\le a\le n$,
so that $h'_1,\ldots,h'_n$ generate $\Omega(M_n)/\mathrm{rad}(\Omega(M_n))$.
By  Remark \ref{rem:stringhoms}, each homomorphism $\Omega(M_n)\to M_n$ is a $k$-linear 
combination of the inclusion homomorphism $\Omega(M_n)\hookrightarrow M_n$
and the homomorphisms $\mu'_{a,b}$ ($2\le a\le n,0\le b\le n$) where 
$\mu'_{a,b}$ sends $h'_a$ to $f_b$ and all other basis elements to zero. Using a similar
argument as for the homomorphisms $\mu_{a,b}$ above, we see that
$\mu'_{a,b}$ factors through a direct sum of copies of $P_0$ for $2\le a\le n,0\le b\le n$. 
Hence $\mathrm{Ext}^1_{\Lambda_{2,1}}(M_n,M_n)\cong k$ for all $n\ge 1$. 
The remaining modules $M((\alpha^{-1}\gamma\beta)^n)$
in $(\ref{eq:c0mods})$ are treated similarly to $M_n$. This proves part (i) and completes the proof of
Proposition \ref{prop:stablend} when $i=2$ and $c=1$.

Suppose now that $i=1$ and $c\in\{0,1\}$.
By \cite[Prop. 3.1]{holm}, for fixed $d\ge 3$, $\Lambda_{1,c}$ and $\Lambda_{2,c}$ are derived 
equivalent. 
By \cite[Cor. 5.5]{rickard1}, this means that there is a stable equivalence of Morita type 
between $\Lambda_{1,c}$ and $\Lambda_{2,c}$. In particular, this 
stable equivalence commutes with $\Omega$ and identifies the stable Auslander-Reiten
quivers of $\Lambda_{1,c}$ and $\Lambda_{2,c}$.  
Thus to complete the proof when 
$i=1$, we need to find the components of the stable Auslander-Reiten quiver of 
$\Lambda_{1,c}$ that correspond under this stable equivalence to the
components of the stable Auslander-Reiten quiver of $\Lambda_{2,c}$
containing modules whose stable endomorphism rings are isomorphic to $k$.
Since the projective $\Lambda_{1,c}$-module cover $P_1$ of $S_1$ is uniserial,
it follows that $S_1$ lies at the boundary of the unique $3$-tube of the
stable Auslander-Reiten quiver of $\Lambda_{1,c}$, which proves part (iii). Let $\mathfrak{C}$ be
a component of the stable Auslander-Reiten quiver of $\Lambda_{1,c}$
containing $Y=S_{001}$.
Since $\mathfrak{C}$ is of type $\mathbb{Z}A_\infty^\infty$ and contains
modules whose stable endomorphism rings have $k$-dimension at least two,
$\mathfrak{C}$ must correspond to one of the components of the stable 
Auslander-Reiten quiver of $\Lambda_{2,c}$ containing $S_1$ and $\Omega(S_1)$. 
This proves part (ii).
If $\mathfrak{C}$ is the component of the stable Auslander-Reiten 
quiver of $\Lambda_{1,c}$ containing $S_0$ (resp. $\Omega(S_0)$), then
$\mathfrak{C}$ is of type $\mathbb{Z}A_\infty^\infty$. Since
we have matched up the components of this type for part (ii), $\mathfrak{C}$ 
must correspond to one of the components of the stable 
Auslander-Reiten quiver of $\Lambda_{2,c}$ containing $S_0$ and $\Omega(S_0)$. 
This proves part (i).
If $c=1$, we also need to find the one-tubes that contain $\Lambda_{1,1}$-modules whose 
stable endomorphism rings are isomorphic to $k$.  Considering all indecomposable 
$\Lambda_{1,1}$-modules lying in one-tubes,
we see that the only such modules whose stable endomorphism rings are possibly isomorphic to $k$
are the modules $S^{(\lambda)}_{010}$ for  $\lambda\in k$. 
Since $\Omega(S^{(\lambda)}_{010})\cong S^{(1-\lambda)}_{010}$
for all $\lambda\in k$, it follows that the stable endomorphism ring of $S^{(\lambda)}_{010}$
is isomorphic to $k$ for all $\lambda\in k$.
This proves part (iv) and completes the proof of Proposition \ref{prop:stablend} when $i=1$ and 
$c\in\{0,1\}$.
\end{proof}


\section{Universal deformation rings}
\label{s:udr}

Assume that $k$, $W$, $F$, $d$, $G$, $B$ and $D$ are as in Hypothesis \ref{hyp:throughout}.
Let $i\in\{1,2\}$ and $c\in\{0,1\}$ such that $B$ is Morita equivalent to $\Lambda_{i,c}$.
As before, $T_0$ (resp. $T_1$) is the simple $B$-module that under this Morita equivalence 
corresponds to the simple $\Lambda_{i,c}$-module $S_0$ (resp. $S_1$).

In this section, we determine the universal deformation rings of all finitely generated indecomposable
$kG$-modules $V$ which belong to $B$ and whose stable endomorphism rings are isomorphic to $k$.
In particular, this will prove Theorem \ref{thm:bigmain}.
We use Proposition \ref{prop:stablend} which, using the Morita equivalence between 
$\Lambda_{i,c}$ and $B$, gives a precise description of these modules.


\subsection{The stable Auslander-Reiten components of $B$ containing $T_0$  and 
$\Omega(T_0)$}
\label{ss:c0}

Let $\mathfrak{C}_0$ be the stable Auslander-Reiten component of $B$ containing 
$T_0$. By Proposition \ref{prop:stablend}(i), the stable endomorphism ring of every module
in $\mathfrak{C}_0\cup\Omega(\mathfrak{C}_0)$ is isomorphic to $k$.
The component $\mathfrak{C}_0$ looks as in Figure \ref{fig:c0comp}, 
where $A_0=T_0$ and for $n\ge 1$, the $B$-module $A_n$ (resp. $A_{-n}$) 
corresponds under the Morita equivalence between $\Lambda_{i,c}$ and $B$  to the 
$\Lambda_{i,c}$-string module $M(C_{i,n})$ (resp. $M(C_{i,-n})$) with
\begin{eqnarray}
\label{eq:c0strings1}
&C_{1,n}=\left(\alpha(\beta^{-1}\gamma^{-1}\alpha^{-1})^{2^{d-2}-1}\beta^{-1}\gamma^{-1}\right)^n,
\quad
C_{1,-n}=\left(\alpha^{-1}(\gamma\beta\alpha)^{2^{d-2}-1}\gamma\beta\right)^n,&\\
\label{eq:c0strings2}
&C_{2,n}=(\alpha\beta^{-1}\gamma^{-1})^n,\quad 
C_{2,-n}=(\alpha^{-1}\gamma\beta)^n.&
\end{eqnarray}
\begin{figure}[ht] \hrule \caption{\label{fig:c0comp} The stable Auslander-Reiten 
component $\mathfrak{C}_0$ near $T_0=A_0$.}
$$\xymatrix @-1.2pc{
&&&&&&\\
&\;\Omega^4(A_2)\;\ar[rd]\ar@{.}[lu]\ar@{.}[ld]\ar@{.}[ru]&&
\;\Omega^2(A_2)\;\ar[rd]\ar@{.}[ru]\ar@{.}[lu]
&&\;A_2\;\ar@{.}[ru]\ar@{.}[rd]\ar@{.}[lu]&\\
&&\;\Omega^2(A_1)\;\ar[rd]\ar[ru]&&\;A_1\;\ar[rd]\ar[ru]&&\\
&\;\Omega^2(A_0)\; \ar[ru]\ar[rd]\ar@{.}[lu]\ar@{.}[ld]&&\;A_0\;\ar[ru]\ar[rd]&&\;\Omega^{-2}(A_0)\;\ar@{.}[ru]\ar@{.}[rd]&\\
&&A_{-1}\ar[ru]\ar[rd]\ar@{.}[ld]&&\Omega^{-1}(A_{-1})\ar[ru]\ar[rd]&&\\
&A_{-2}\ar[ru]\ar@{.}[ld]\ar@{.}[lu]\ar@{.}[rd]&&
\Omega^{-2}(A_{-2})\ar[ru]\ar@{.}[rd]\ar@{.}[ld]
&&\Omega^{-4}(A_{-2})\ar@{.}[rd]\ar@{.}[ru]\ar@{.}[ld]&\\
&&&&&&\\&&&&&&
}$$
\hrule
\end{figure}
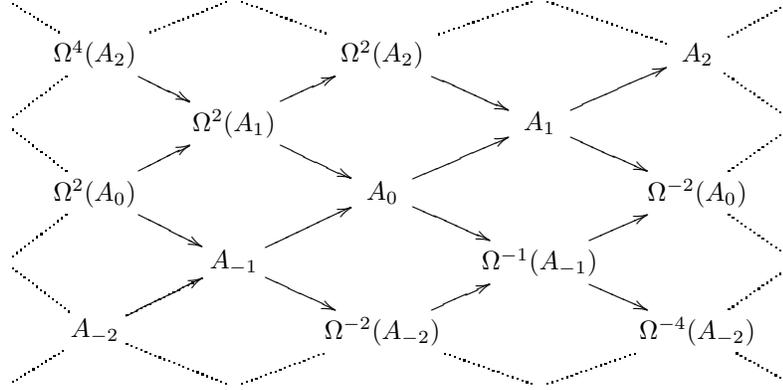

\begin{lemma}
\label{lem:liftingc0}
For all $n\in\mathbb{Z}$, the $B$-module $A_n$  from Figure $\ref{fig:c0comp}$
has $2$ non-isomorphic lifts over $W$.
\end{lemma}

\begin{proof}
If $B$ is Morita equivalent to $\Lambda_{2,c}$, the ideas to prove Lemma 
\ref{lem:liftingc0} are similar to the ones used to prove the corresponding fact in 
\cite[Lemma 4.2]{blello}. Since the proof of \cite[Lemma 4.2]{blello} uses some
explicit lifts of $kS_4$-modules over $W$, which we do not know how to establish
for more general blocks $B$, we provide a slightly different strategy, which will also work
if $B$ is Morita equivalent to $\Lambda_{1,c}$.

Suppose first that $B$ is Morita equivalent to $\Lambda_{2,c}$.
We will prove that for all $n\ge 0$, the module $A_n$ from Figure $\ref{fig:c0comp}$
has $2$ non-isomorphic lifts over $W$, the case of the modules $A_{-n}$ being similar. 
Let $P_{T_0}^W$ be a lift over $W$ of the projective $B$-module cover
${P_{T_0}}$ of $T_0$. It follows from
the decomposition matrix in Figure \ref{fig:decomp} that 
$F\otimes_W P_{T_0}^W = \Xi_1\oplus \Xi_2\oplus \Xi_3\oplus \Xi_4$ 
where $\Xi_j$ is a simple $FG$-module with $F$-character $\chi_j$ for $1\le j\le 4$. 
By \cite[Prop. (23.7)]{CR}, $X=P_{T_0}^W/(P_{T_0}^W\cap (\Xi_2\oplus \Xi_4))$ 
is a $WG$-module which is free over $W$  and has $F$-character $\chi_1+\chi_3$, 
and $X'=P_{T_0}^W/(P_{T_0}^W\cap (\Xi_1\oplus \Xi_3))$ is a $WG$-module
which is free over $W$  and has $F$-character $\chi_2+\chi_4$. Moreover, both
$k\otimes_W X$ and $k\otimes_W X'$ are indecomposable $kG$-modules
that are quotient modules of $P_{T_0}$ and have composition factors $T_0,T_0,T_1$. 
Since it follows from the decomposition matrix in Figure \ref{fig:decomp} and  
\cite[Prop. (23.7)]{CR} that the uniserial $kG$-module $T_{1\cdots 1}$ of length
$2^{d-2}$ whose composition factors are all isomorphic to
$T_1$ does not have a lift over $W$, neither $k\otimes_W X$ nor 
$k\otimes_W X'$ can be isomorphic to $\Omega(T_{1\cdots 1})$. 
This means that $k\otimes_W X$ (resp. $k\otimes_W X'$) corresponds either to
the uniserial  $\Lambda_{2,c}$-string module $M(\beta^{-1}\gamma^{-1})$, or to the 
$\Lambda_{2,c}$-string module $M(\alpha\beta^{-1})$, or to the 
$\Lambda_{2,c}$-band module $M(\alpha\beta^{-1}\gamma^{-1},\lambda,1)$
for some $\lambda\in k^*$. 
In any case, we have non-split short exact sequences of $kG$-modules
\begin{eqnarray}
\label{eq:anotherses1}
\xymatrix{
0\ar[r]& A_{n-1} \ar[r] & A_n \ar[r] & k\otimes_W X\ar[r]& 0},\\
\label{eq:anotherses2}
\xymatrix{
0\ar[r]& A_{n-1} \ar[r] & A_n \ar[r] & k\otimes_W X'\ar[r]& 0}.
\end{eqnarray}
It follows from the decomposition 
matrix in Figure \ref{fig:decomp} that $A_0=T_0$ has two non-isomorphic lifts 
$(U_{0,1}^W,\phi_{0,1})$ and 
$(U_{0,2}^W,\phi_{0,2})$ over $W$ whose $F$-characters are given by $\chi_1$ and $\chi_2$, 
respectively. Let $n\ge 1$.
Using that $A_0$ corresponds to the $\Lambda_{2,c}$-module $S_0$ and, for $n>1$, 
$A_{n-1}$ corresponds to the $\Lambda_{2,c}$-module
$M(C_{2,n-1})$, where $C_{2,n-1}$ is the string in $(\ref{eq:c0strings2})$, we can use
Remark \ref{rem:stringhoms} to prove
$$\mathrm{Hom}_{kG}(k\otimes_W X,A_{n-1})\cong k^{n}\cong 
\mathrm{Hom}_{kG}(k\otimes_W X',A_{n-1})$$ and
$$\mathrm{Ext}^1_{kG}(k\otimes_W X,A_{n-1})\cong k \cong
\mathrm{Ext}^1_{kG}(k\otimes_W X',A_{n-1}).$$
Assume by induction that $A_{n-1}$ has two lifts $(U_{n-1,1}^W,\phi_{n-1,1})$ and 
$(U_{n-1,2}^W,\phi_{n-1,2})$
over $W$ such that the $F$-character of  $U_{n-1,1}^W$ is
\begin{equation}
\label{eq:char1}
\left\{\begin{array}{ll} 
\chi_1+s(\chi_1+\chi_3) + s(\chi_2+\chi_4) & \mbox{ if }n-1=2s\\
\chi_1+ s(\chi_1+\chi_3) + (s+1)(\chi_2+\chi_4)& \mbox{ if }n-1=2s+1
\end{array}\right\},
\end{equation}
and the $F$-character of $U_{n-1,2}^W$ is 
\begin{equation}
\label{eq:char2}
\left\{\begin{array}{ll} 
\chi_2+s(\chi_1+\chi_3) + s(\chi_2+\chi_4) & \mbox{ if }n-1=2s\\
\chi_2+ (s+1)(\chi_1+\chi_3) + s(\chi_2+\chi_4)& \mbox{ if }n-1=2s+1
\end{array}\right\}.
\end{equation}
Then if $n-1=2s$, 
$$\mathrm{dim}_F\,\mathrm{Hom}_{FG}(F\otimes_WX',F\otimes_W U_{n-1,1}^W)
=s+s=2s=n-1,$$
and if $n-1=2s+1$,
$$\mathrm{dim}_F\,\mathrm{Hom}_{FG}(F\otimes_WX,F\otimes_W U_{n-1,1}^W)
=1+s+s=2s+1=n-1.$$
Similarly, if $n-1=2s$ then 
$\mathrm{dim}_F\,\mathrm{Hom}_{FG}(F\otimes_WX,F\otimes_W U_{n-1,2}^W)=n-1$, 
and if $n-1=2s+1$ then
$\mathrm{dim}_F\,\mathrm{Hom}_{FG}(F\otimes_WX',F\otimes_W 
U_{n-1,2}^W)=n-1$.
Hence by \cite[Lemma 2.3.2]{3sim}, $A_n$ has two non-isomorphic lifts 
$(U_{n,1}^W,\phi_{n,1})$ and $(U_{n,2}^W,\phi_{n,2})$ over $W$ whose $F$-characters are as in 
$(\ref{eq:char1})$ and $(\ref{eq:char2})$ when $n-1$ is replaced by $n$. 

Now suppose that $B$ is Morita equivalent to $\Lambda_{1,c}$. 
Again, we concentrate on the modules $A_n$ for $n\ge 0$, the case of the modules $A_{-n}$
being similar. We want to use a similar inductive argument as above. The tricky part
is to find modules that correspond to the above $X$ and $X'$ and which
provide short exact sequences of the form $(\ref{eq:anotherses1})$ and
$(\ref{eq:anotherses2})$. For this, we use a variation of \cite[Lemma 2.3.2]{3sim}. 
Consider the string
$$Z=(\beta^{-1}\gamma^{-1}\alpha^{-1})^{2^{d-2}-1}\beta^{-1}.$$
Analyzing the $\Lambda_{1,c}$-modules that are non-trivial extensions of $M(Z)$ 
by $S_0$, we see that they are isomorphic to either
$$M(Z\gamma^{-1}), \mbox{ or } M(\alpha Z), \mbox{ or } 
M(\alpha Z\gamma^{-1},\lambda,1) \mbox{ for $\lambda\in k^*$. }$$
Moreover, if $M$ is any of these modules, then for all $n\ge 1$,
there is a short exact sequence of $\Lambda_{1,c}$-modules
$$0\to M(C_{1,n-1})\to M(C_{1,n}) \to M\to 0,$$
where we set $M(C_{1,0})=S_0$. Let $Y$ be the $B$-module corresponding to
the $\Lambda_{1,c}$-module $M(Z)$. 
Using the projective $B$-module cover $P_{T_1}$ of $T_1$, 
it follows from the decomposition matrix for $B$ in Figure \ref{fig:decomp}
and \cite[Prop. (23.7)]{CR} that $Y$ has two non-isomorphic
lifts $(Y^W_1,\psi_1)$ and $(Y^W_2,\psi_2)$ over $W$ with $F$-characters 
$\chi_3+\sum_{j=1}^{2^{d-2}-1}\chi_{5,j}$
and $\chi_4+\sum_{j=1}^{2^{d-2}-1}\chi_{5,j}$, respectively. 
Also, $T_0$ has two non-isomorphic lifts $(U^W_{0,1},\phi_{0,1})$ and $(U^W_{0,2},\phi_{0,2})$ with
$F$-characters $\chi_1$ and $\chi_2$, respectively. This implies
that $\mathrm{Hom}_{FG}(F\otimes_W Y^W_\ell,F\otimes_W U^W_{0,\ell})=0$
for $\ell\in\{1,2\}$. Moreover, $\mathrm{Hom}_{kG}(Y,T_0)\cong k$.
Following the proof of \cite[Lemma 2.3.2]{3sim}, this implies that, for $\ell\in\{1,2\}$, the natural map
$$\mathrm{Ext}^1_{WG}(Y^W_\ell,U^W_{0,\ell})\to\mathrm{Ext}^1_{kG}(Y,T_0),$$
which is induced by reduction modulo $2$, has non-trivial image. 
Therefore, there exist $WG$-modules $X$ and $X'$ that are free over $W$
such that $k\otimes_W X$ and $k\otimes_W X'$ lie in
short exact sequences of the form $(\ref{eq:anotherses1})$ and
$(\ref{eq:anotherses2})$, respectively,
and such that the $F$-character of $X$ is equal to 
$\chi_1+\chi_3+\sum_{j=1}^{2^{d-2}-1}\chi_{5,j}$ and the $F$-character of $X'$
is equal to $\chi_2+\chi_4+\sum_{j=1}^{2^{d-2}-1}\chi_{5,j}$. We can now
proceed using a similar inductive argument as in the case when $B$ is
Morita equivalent to $\Lambda_{2,c}$ to complete the proof of Lemma \ref{lem:liftingc0}.
\end{proof}

\begin{prop}
\label{prop:udrc0}
Let $V$ be a $B$-module in $\mathfrak{C}_0\cup\Omega(\mathfrak{C}_0)$. 
Then $R(G,V)$ is isomorphic to a $W$-subalgebra of $W[\mathbb{Z}/2]$. 
If $B$ is a principal block, then $R(G,V)\cong W[\mathbb{Z}/2]$.
\end{prop}

\begin{proof}
By Proposition \ref{prop:stablendudr}(ii), it suffices to prove Proposition \ref{prop:udrc0} when
$V=A_n$ for some $n\in\mathbb{Z}$, where $A_n$ is as in Figure \ref{fig:c0comp}.

Suppose that $B$ is Morita equivalent to $\Lambda_{2,c}$.
Let $n\ge 0$ and $V_n\in\{A_n,A_{-n}\}$.
We first prove that for all $n\ge 0$, $\overline{R}_n=R(G,V_n)/2R(G,V_n)$ is isomorphic to $k[t]/(t^2)$. 
From Proposition \ref{prop:stablend}(i) it follows that $\mathrm{Ext}^1_{kG}(V_n,V_n)\cong k$, 
which implies that $\overline{R}_n$ is isomorphic to a quotient algebra of $k[[t]]$.
Let $T_{00}$ be the uniserial $kG$-module with composition factors $T_0,T_0$
and let $P_{T_0}$ be the projective $B$-module cover of $T_0$.
Then we have a non-split short exact sequence of $kG$-modules
\begin{equation}
\label{eq:ses}
\xymatrix{
0\ar[r] & V_n \ar[r] & (P_{T_0})^n \oplus T_{00}\ar[r] & V_n\ar[r] &0}.
\end{equation}
Hence $\overline{U}_n=(P_{T_0})^n \oplus T_{00}$ is a free $k[t]/(t^2)$-module 
which defines a lift $(\overline{U}_n,\overline{\phi}_n)$ of $V_n$ over $k[t]/(t^2)$. Thus there is a unique 
$k$-algebra homomorphism
$\sigma:\overline{R}_n\to k[t]/(t^2)$ corresponding to $(\overline{U}_n,\overline{\phi}_n)$. 
Since $(\overline{U}_n,\overline{\phi}_n)$ is not the trivial lift of $V_n$ over
$k[t]/(t^2)$, $\sigma$ is surjective. We want to show that $\sigma$ is a $k$-algebra 
isomorphism. Suppose this is false. Then there exists a surjective $k$-algebra 
homomorphism $\sigma':\overline{R}_n\to k[t]/(t^3)$
such that $\pi\circ\sigma'=\sigma$ where $\pi:k[t]/(t^3)\to k[t]/(t^2)$ is the natural projection. Let 
$(\overline{E}_n,\eta_n)$ be a lift of $V_n$ over $k[t]/(t^3)$ relative to $\sigma'$. Then 
$k[t]/(t^2)\otimes_{\pi}\overline{E}_n\cong \overline{U}_n$ and
$t^2\overline{E}_n\cong V_n$. Thus we have a short exact sequence of
$k[t]/(t^3)G$-modules
\begin{equation}
\label{eq:thes}
0\to t^2\overline{E}_n\to \overline{E}_n\to \overline{U}_n\to 0.
\end{equation}
Since $\mathrm{Ext}^1_{kG}(\overline{U}_n,t^2 \overline{E}_n)
\cong \mathrm{Ext}^1_{kG}(T_{00},V_n)=0$,
the sequence $(\ref{eq:thes})$ splits as a sequence of $kG$-modules,
i.e. $\overline{E}_n\cong V_n\oplus \overline{U}_n$ as $kG$-modules. 
Identifying the $kG$-module structure of $\overline{E}_n$ with 
$V_n\oplus \overline{U}_n$,
multiplication by $t$ on  an element $(v,u)\in V_n\oplus \overline{U}_n=\overline{E}_n$ is given by
$t\cdot (v,u) = (g(u),t\cdot u)$,
where $g:\overline{U}_n\to V_n$ is a surjective $kG$-module homomorphism.
Hence
$t^2\cdot (v,u) = (g(t\cdot u),t^2\cdot u)=(g(t\cdot u),0).$
Since $t^2\overline{E}_n\cong V_n$, this means that $g\big|_{t\cdot \overline{U}_n}:
t\cdot \overline{U}_n\to V_n$ must be 
an isomorphism of $kG$-modules.
On the other hand, 
$\overline{U}_n=(P_{T_0})^n \oplus T_{00}$ and $\mathrm{soc}(t\cdot \overline{U}_n)$ 
contains a non-zero submodule of $\mathrm{soc}((P_{T_0})^n)$ (resp. $\mathrm{soc}(T_{00})$) if
$n\ge 1$ (resp. $n=0$). Since every surjective $kG$-module 
homomorphism $\overline{U}_n\to V_n$ sends the elements in 
$\mathrm{soc}((P_{T_0})^n)$ (resp. $\mathrm{soc}(T_{00})$) to zero if $n\ge 1$ (resp. $n=0$), 
we get a contradiction.
Therefore, $\overline{E}_n$ does not exist, which means that $\sigma$ is a 
$k$-algebra isomorphism. Thus $\overline{R}_n=R(G,V_n)/2R(G,V_n)\cong k[t]/(t^2)$.

By \cite[Prop. 3.1]{holm}, for fixed $d\ge 3$, $\Lambda_{1,c}$ and $\Lambda_{2,c}$ are derived 
equivalent. 
By \cite[Cor. 5.5]{rickard1}, this means that there is a stable equivalence of Morita type 
between $\Lambda_{1,c}$ and $\Lambda_{2,c}$. Hence it follows from \cite[Lemma 2.2.3]{3sim}
that $R(G,V_n)/2R(G,V_n)\cong k[t]/(t^2)$ for all $n\ge 0$
when $B$ is Morita equivalent to $\Lambda_{1,c}$ and $V_n\in\{A_n,A_{-n}\}$ 
for $B$.

Suppose now that $B$ is Morita equivalent to either $\Lambda_{1,c}$ or
$\Lambda_{2,c}$, and let $n\ge 0$. By Lemma \ref{lem:liftingc0}, $V_n$ has two 
non-isomorphic lifts over $W$. By \cite[Lemma 2.1]{bc5},
this means that $R(G,V_n)\cong W[[t]]/(t(t-2\mu_{V_n}))$ for 
some non-zero $\mu_{V_n}\in W$. This implies that 
$R(G,V_n)$ is isomorphic to a $W$-subalgebra of $W[\mathbb{Z}/2]$.

To prove the last statement of Proposition \ref{prop:udrc0},
let us now assume that $B$ is a principal block of $kG$. 
By Theorem \ref{thm:principalblocks}(i), there exists $i\in\{1,2\}$
such that $B$ is Morita equivalent to $\Lambda_{i,0}$, i.e. the parameter $c$ is zero.
As above, let $T_{00}$ be a uniserial $kG$-module with composition factors $T_0,T_0$. 
By Theorem \ref{thm:principalblocks}(ii), there exists an element $\tau\in G$ of order $2$ such 
that $\mathrm{Res}^G_{\langle\tau\rangle}(T_{00})\cong k\langle \tau\rangle$.
Let $J=\langle \tau\rangle$. Then $\mathrm{Res}^G_J(T_{00})\cong kJ$.
Consider the restriction of the projective cover $P_{T_0}$ to $J$, which is a
projective $kJ$-module. If $B$ is Morita equivalent to $\Lambda_{2,0}$, we can directly use
the description of the radical series of $P_{T_0}$  in Figure \ref{fig:projs} to see
that $\mathrm{Res}_J^G(T_1)$ is also a projective $kJ$-module.
If $B$ is Morita equivalent to $\Lambda_{1,0}$, then we can use
the description of the radical series of $P_{T_0}$ in Figure \ref{fig:projs} to see that
$\mathrm{Res}_J^G(\Omega^{-1}(T_{001}))$ is a projective $kJ$-module, where
$T_{001}$ denotes the uniserial $kG$-module with descending composition factors
$T_0,T_0,T_1$. Since $\mathrm{Res}_J^G(P_{T_1})$ is projective, this then implies
that $\mathrm{Res}_J^G(T_{001})$, and hence $\mathrm{Res}_J^G(T_1)$, 
is a projective $kJ$-module. It follows that for both $i=1$ and $i=2$ and for all $n\ge 0$,
$\mathrm{Res}_J^G(V_n)\cong k\oplus Q_n$, where $Q_n$ is a projective 
$kJ$-module.
Therefore, $\mathrm{Res}^G_J(V_n)$ is a $kJ$-module whose
stable endomorphism ring is isomorphic to $k$. Moreover, its universal deformation ring is
$R(J,\mathrm{Res}^G_J(V_n))\cong W[\mathbb{Z}/2]$.
Let $(U_{n,J},\phi_{n,J})$  be a universal lift of $\mathrm{Res}^G_J(V_n)$ over 
$W[\mathbb{Z}/2]$,
and let $(U_n,\phi_n)$ be a universal lift of $V_n$ over $R(G,V_n)$. 
Then there exists a unique $W$-algebra homomorphism 
$f_n:W[\mathbb{Z}/2]\to R(G,V_n)$ in $\mathcal{C}$  such that the lift 
$(\mathrm{Res}^G_{J}(U_n),\mathrm{Res}^G_{J}(\phi_n))$ 
is isomorphic to the lift 
$(R(G,V_n)\otimes_{W[\mathbb{Z}/2],f_n} U_{n,J},\,(\phi_{n,J})_{f_n})$. 
Since for both $i=1$ and $i=2$, $\overline{U}_n=(P_{T_0})^n \oplus T_{00}$ defines a non-trivial lift 
of $V_n$ over $k[\epsilon]/(\epsilon^2)$ and the restriction 
$\mathrm{Res}^G_J(\overline{U}_n)$ defines a non-trivial lift of $\mathrm{Res}^G_J(V_n)$ over 
$k[\epsilon]/(\epsilon^2)$, it follows that $f_n:W[\mathbb{Z}/2]\to R(G,V_n)$ is surjective.
By Lemma \ref{lem:liftingc0}, there are two distinct morphisms
$R(G,V_n)\to W$ in $\mathcal{C}$, which means that $\mathrm{Spec}(R(G,V_n))$ 
contains both points of the generic fiber of $\mathrm{Spec}(W[\mathbb{Z}/2])$.
Since the Zariski closure of these points is all of $\mathrm{Spec}(W[\mathbb{Z}/2])$,
this implies that $R(G,V_n)$ must be isomorphic to $W[\mathbb{Z}/2]$.
This completes the proof of Proposition \ref{prop:udrc0}.
\end{proof}


\subsection{The $3$-tube of the stable Auslander-Reiten quiver of $B$}
\label{ss:3tube}

Let $\mathfrak{T}_3$ be the stable Auslander-Reiten component of $B$ 
that is a $3$-tube.
By Proposition \ref{prop:stablend}(iii), the only modules in $\mathfrak{T}_3$
whose stable endomorphism rings are isomorphic to $k$ are the modules at its boundary.

\begin{prop}
\label{prop:3tube}
Let $V$ be a $B$-module at the boundary of  $\mathfrak{T}_3$. Then
$R(G,V)\cong k$.
\end{prop}

\begin{proof}
We want to use similar arguments to the ones used in \cite[Sect. 5.2]{3sim}. 
The crucial step is to establish that certain statements are true for $\mathfrak{T}_3$
that closely resemble the statements in \cite[Facts 5.2.1]{3sim}. 
Because of the work in \cite[Chapter V]{erd}, we indeed have the following facts.
Let $V$ be a module at the boundary of $\mathfrak{T}_3$ and let $K$ be a vertex of $V$.
\begin{enumerate}
\item[(i)] The group $K$ is a Klein four group and
the quotient group $N_G(K)/C_G(K)$ is isomorphic to a symmetric group $S_3$.
\item[(ii)] There is a block $b$ of $kN_G(K)$ with $b^G=B$ such that the Green 
correspondent $fV$ of $V$ belongs to the boundary of a $3$-tube in the stable
Auslander-Reiten quiver of $b$. Moreover, $b$ is Morita equivalent to $kS_4$ modulo 
the socle.
\end{enumerate}

Because of these facts, it follows as in \cite[Prop. 5.2.4]{3sim} that if $U$ is an 
indecomposable $b$-module belonging to the boundary of the $3$-tube of the stable 
Auslander-Reiten quiver of $b$, then $U$ has a universal deformation ring 
that is isomorphic to $k$.
Arguing as in the proof of \cite[Cor. 5.2.5]{3sim}, we then see that this implies
that $R(G,V)\cong k$.
\end{proof}


\subsection{The stable Auslander-Reiten components of $B$ corresponding to part (ii) of
Proposition \ref{prop:stablend}}
\label{ss:omegaorbit}

Let $\mathfrak{C}_1$ be one of the stable Auslander-Reiten components of $B$ that 
correspond to part (ii) of Proposition \ref{prop:stablend}. 
Let $T_{100}$ (resp. $T_{001}$) be the uniserial $B$-module with descending composition factors 
$T_1,T_0,T_0$ (resp. $T_0,T_0,T_1$). Note that $T_{001}\cong \Omega^2(T_{100})$.
If $B$ is Morita
equivalent to $\Lambda_{i,c}$, then $\mathfrak{C}_1$ contains the $B$-module $Z$, where
\begin{equation}
\label{eq:z12}
\mbox{$Z=T_{100}\;$ if $i=1$} \quad \mbox{ (resp.  $Z=T_1\;$ if $i=2$).}
\end{equation}
By Proposition \ref{prop:stablend}(ii), the only modules in $\mathfrak{C}_1\cup\Omega(\mathfrak{C}_1)$ 
whose stable endomorphism rings are isomorphic to $k$ lie in the $\Omega$-orbit of $Z$. 

\begin{prop}
\label{prop:omegaorbit}
Let $V$ be a $B$-module in $\mathfrak{C}_1\cup\Omega(\mathfrak{C}_1)$ whose stable
endomorphism ring is isomorphic to $k$. Then $R(G,V)\cong W[[t]]/(t\, p_d(t),2\,p_d(t))$, 
where $p_d(t)\in W[t]$ is as in Definition $\ref{def:seemtoneed}$. In particular, $R(G,V)$ is 
isomorphic to a subquotient algebra of $WD$.
\end{prop}

\begin{proof}
By Proposition \ref{prop:stablendudr}(ii) and Proposition \ref{prop:stablend}(ii), 
it suffices to prove Proposition \ref{prop:omegaorbit} when $V=Z$ as in (\ref{eq:z12}).

Suppose first that $B$ is Morita equivalent to $\Lambda_{2,c}$. Then $Z= T_1$.
We first prove that $\overline{R}=R(G,T_1)/2R(G,T_1)$ is isomorphic to $k[t]/(t^{2^{d-2}})$.
By Proposition \ref{prop:stablend}(ii), $\mathrm{Ext}^1_{kG}(T_1,T_1)\cong k$, which
implies that $\overline{R}$ is isomorphic to a quotient algebra of $k[[t]]$. 
By Figure \ref{fig:projs}, the projective indecomposable $kG$-module $P_{T_1}$ has the form 
$\begin{array}{c}T_1\\U_1\;U_2\\T_1\end{array}$ where $U_1$ is uniserial of length $2$
and $U_2$ is uniserial of length $2^{d-2}-1$ whose composition factors are all
isomorphic to $T_1$. Let $\overline{U}$ be the uniserial $B$-module corresponding to the
$\Lambda_{2,c}$-string module $M(\eta^{2^{d-2}-1})$ and let $T_{001}$ be the
uniserial $B$-module corresponding to the $\Lambda_{2,c}$-string module $M(\beta\alpha)$.
Then $\overline{U}=\begin{array}{c}T_1\\U_2\end{array}$ and 
$T_{001}=\begin{array}{c}U_1\\T_1\end{array}$, and
$$\mathrm{Ext}^1_{kG}(\overline{U},T_1)=
\underline{\mathrm{Hom}}_{kG}(\Omega(\overline{U}),T_1)=
\underline{\mathrm{Hom}}_{kG}(T_{001},T_1)=0.$$
By \cite[Lemma 2.3.1]{3sim}, this implies that $\overline{R}\cong k[t]/(t^{2^{d-2}})$ 
and that the universal mod $2$ deformation of $T_1$ over $\overline{R}$ is 
represented by the $kG$-module $\overline{U}$. 

To compute the universal deformation ring $R=R(G,T_1)$,
we use similar arguments to the ones used in \cite[Sect. 5.3]{3sim}.
The uniserial $kG$-module $\overline{U'}=\overline{U}/T_1$ of length $2^{d-2}-1$ defines a lift of $T_1$ over $k[t]/(t^{2^{d-2}-1})$. It follows from the decomposition matrix in 
Figure \ref{fig:decomp} and \cite[Prop. (23.7)]{CR} that there is a $W$-pure 
$WG$-sublattice $X'$ of the projective indecomposable $WG$-module $P_{T_1}^W$ 
with top $T_1$ such that $U'=P_{T_1}^W/X'$ has $F$-character
$\sum_{\ell=2}^{d-1}\rho_\ell=\sum_{j=1}^{2^{d-2}-1}\chi_{5,j}$.
Since $U'/2U'$ is an indecomposable $kG$-module with top $T_1$ that has the same 
composition factors as $\overline{U'}$, it is isomorphic to $\overline{U'}$. 
Hence $U'$ defines a lift $(U',\phi')$ of $\overline{U'}$ over $W$. 
Using similar arguments as in the proof of \cite[Thm. 5.1]{3sim} and employing
the statements about the ordinary characters belonging to $B$ discussed
in Section \ref{ss:ordinary}, it follows that $U'$ also defines a lift
$(U',\psi')$ of $T_1$ over $R'$, where $R'=W[[t]]/(p_d(t))$ is as in Definition 
\ref{def:seemtoneed}. We 
therefore have a continuous $W$-algebra homomorphism $\sigma':R\to R'$ relative to $(U',\psi')$. 
Since $U'/2U'$ is indecomposable as a $kG$-module, $\sigma'$ must be surjective. 
By \cite[Lemma 2.3.3]{3sim},  it follows that 
$R\cong W[[t]]/(p_d(t)(t-2\mu),a\,2^mp_d(t))$ for certain $\mu\in W$, $a\in\{0,1\}$ 
and $0 < m\in\mathbb{Z}$. If $a=0$, then $R\cong W[[t]]/(p_d(t)(t-2\mu))$ is free over $W$.
If $a=1$, then $(W/2^mW)\otimes_W R\cong (W/2^m W)[[t]]/(p_d(t) (t-2\mu))$ is free over $W/2^mW$.
Therefore it follows that if $a=0$ (resp. $a=1$), then there is 
a lift of $\overline{U}$, when regarded as a $kG$-module, over $W$ 
(resp. over $W/2^mW$). But $\overline{U}\cong \Omega^{-1}(T_{001})$, which means
that $\overline{U}$ is a module at the end of the $3$-tube, and hence
$R(G,\overline{U})\cong k$ by Proposition \ref{prop:3tube}. Therefore,
$a=1$ and $m=1$. Since by \cite[Lemma 2.3.6]{3sim} the ring $W[[t]]/(t\,p_d(t),2\,p_d(t))$
is isomorphic to a subquotient algebra of $WD$, this proves Proposition \ref{prop:omegaorbit} 
when $i=2$.

Suppose now that $B$ is Morita equivalent to $\Lambda_{1,c}$. Then 
$Z=T_{100}$.
Since by \cite[Prop. 3.1]{holm} and \cite[Cor. 5.5]{rickard1}, 
there is a stable equivalence of Morita type between $\Lambda_{1,c}$ and 
$\Lambda_{2,c}$, it follows by \cite[Lemma 2.2.3]{3sim} that 
$R(G,T_{100})/2R(G,T_{100})\cong k[t]/(t^{2^{d-2}})$. 
By considering the projective cover $P_{T_1}$ of $T_1$, it follows
moreover that $\overline{U}=\Omega^{-1}(T_1)$ defines the universal mod $2$ deformation 
of $T_{100}$.
The uniserial $kG$-module $\overline{U'}=\overline{U}/T_{100}$ of length 
$3(2^{d-2}-1)$ defines a lift of $T_{100}$ over $k[t]/(t^{2^{d-2}-1})$. It follows from the 
decomposition matrix in Figure \ref{fig:decomp} and \cite[Prop. (23.7)]{CR} that there 
is a $W$-pure $WG$-sublattice $X'$ of the projective indecomposable $WG$-module 
$P_{T_1}^W$ with top $T_1$ such that $U'=P_{T_1}^W/X'$ has $F$-character
$\sum_{\ell=2}^{d-1}\rho_\ell=\sum_{j=1}^{2^{d-2}-1}\chi_{5,j}$.
Since $U'/2U'$ is an indecomposable $kG$-module with top $T_1$ that has the same 
composition factors as $\overline{U'}$, it is isomorphic to $\overline{U'}$. 
Using similar arguments to the case when $B$ is Morita equivalent to $\Lambda_{2,c}$,
we see that $R(G,T_{100})\cong W[[t]]/(t\,p_d(t),2\,p_d(t))$. This proves Proposition
\ref{prop:omegaorbit} when $i=1$.
\end{proof}


\subsection{The one-tubes of the stable Auslander-Reiten quiver of $B$ in the case when $c=1$}
\label{ss:onetubes}

In this subsection, we assume $c=1$, i.e. $B$ is Morita equivalent to $\Lambda_{i,1}$
for some $i\in\{1,2\}$.
For $\lambda\in k$, let $\mathfrak{T}_{B,\lambda}$ be the one-tube of the stable Auslander-Reiten
quiver of $B$ that corresponds, under the Morita equivalence, to the one-tube 
$\mathfrak{T}^{(\lambda)}$ of the stable Auslander-Reiten quiver of $\Lambda_{i,1}$ from part (iv) 
of Proposition \ref{prop:stablend}. By Proposition \ref{prop:stablend}(iv), the only modules in 
$\mathfrak{T}_{B,\lambda}$ whose stable endomorphism rings are isomorphic to $k$ are the 
modules at its boundary.

\begin{prop}
\label{prop:onetubes}
Let $\lambda\in k$ and let $V$ be a $B$-module at the boundary of $\mathfrak{T}_{B,\lambda}$. 
Then $R(G,V)\cong W[[t]]/(2\, f_V(t))$ for a certain power series $f_V(t)\in W[[t]]$.
\end{prop}

\begin{proof}
Suppose first that $B$ is Morita equivalent to $\Lambda_{2,1}$. For $\lambda\in k$, let
$V_\lambda$ be the $B$-module which under the Morita equivalence corresponds to
the $\Lambda_{2,1}$-module $S^{(\lambda)}_{010}$ from Definition \ref{def:modules}.
We first prove that $\overline{R}_\lambda=
R(G,V_\lambda)/2 R(G,V_\lambda)$ is isomorphic to $k[[t]]$. Since 
$\mathrm{Ext}^1_{kG}(V_\lambda,V_\lambda)\cong k$ by Proposition \ref{prop:stablend}(iv), it follows that 
$\overline{R}_\lambda$ is isomorphic to a quotient algebra of $k[[t]]$. For $n\ge 1$ and $\lambda\in k$,
define $S^{(\lambda,n)}_{010}=M(\alpha\beta^{-1}\gamma^{-1},\lambda,n)$ for $\lambda\ne 0$ and
$S^{(0,n)}_{010}=M(\beta^{-1}\gamma^{-1}(\alpha\beta^{-1}\gamma^{-1})^{n-1})$, and let $V_{\lambda,n}$ be the 
$B$-module corresponding to the $\Lambda_{2,1}$-module $S^{(\lambda,n)}_{010}$.
Using the description of $S^{(\lambda,n)}_{010}$ via canonical bases from Definitions
\ref{def:strings} and \ref{def:bands} together with the Morita equivalence between $B$ and
$\Lambda_{2,1}$, we can use induction to establish the following for all $n\ge 2$ and $\lambda\in k$.
There are non-split short exact sequences of $B$-modules of the form
\begin{eqnarray}
\label{eq:onetubes1}
&0 \to V_\lambda \xrightarrow{\iota_{1,n}} V_{\lambda,n} \xrightarrow{\pi_{n,n-1}} V_{\lambda,n-1} \to 
0,& \\
\label{eq:onetubes2}
&0 \to V_{\lambda,n-1} \xrightarrow{\iota_{n-1,n}} V_{\lambda,n} \xrightarrow{\pi_{n,1}} V_\lambda \to 
0,&
\end{eqnarray}
such that $\pi_{n,1} = \pi_{2,1}\circ\cdots\circ\pi_{n,n-1}$ and 
$\iota_{1,n}=\iota_{n-1,n}\circ\cdots\circ\iota_{1,2}$. Moreover, if 
$\tau_{\lambda,n}=\iota_{n-1,n}\circ \pi_{n,n-1}$, then for all $1\le \ell\le n-1$, 
$\left(\tau_{\lambda,n}\right)^\ell(V_{\lambda,n})=\iota_{n-\ell,n}(V_{\lambda,n-\ell})$, where
$\iota_{n-\ell,n} = \iota_{n-1,n}\circ\cdots\circ\iota_{n-\ell,n-\ell+1}$, and 
$\left(\tau_{\lambda,n}\right)^n(V_{\lambda,n})=0$. It follows that $\tau_{\lambda,n}$ defines 
a $k[t]/(t^n)$-module structure on $V_{\lambda,n}$ such that the submodules
$t^\ell\cdot V_{\lambda,n}$, $0\le \ell\le n$, form  a filtration of $V_{\lambda,n}$
of $k[t]/(t^n)$-modules whose successive quotients are isomorphic to $V_\lambda$.
We conclude that $V_{\lambda,n}$ is a free $k[t]/(t^n)$-module.
Since $V_{\lambda,n}/t\cdot V_{\lambda,n}\cong V_\lambda$ as $B$-modules, 
it follows that $V_{\lambda,n}$
defines a lift of $V_\lambda$ over $k[t]/(t^n)$. Hence there exists a unique $k$-algebra homomorphism
$\sigma_{\lambda,n}:\overline{R}_\lambda\to k[t]/(t^n)$ corresponding to this lift. Since $V_{\lambda,n}$
is an indecomposable $kG$-module, it follows that $\sigma_{\lambda,n}$ is surjective
for all $n\ge 1$. But this implies that $\overline{R}_\lambda\cong k[[t]]$. 
Since $\Omega^2(V_\lambda)\cong V_\lambda$, we have
$\mathrm{Ext}^2_{kG}(V_\lambda,V_\lambda)\cong k$. By \cite[Sect. 1.6]{maz1}, this means
that $R(G,V_\lambda)$ is isomorphic to a quotient algebra of $W[[t]]$ by an ideal generated
by a single element. In other words, $R(G,V_\lambda)\cong W[[t]]/(h_\lambda(t))$ for some 
$h_\lambda(t)\in W[[t]]$. Since $R(G,V_\lambda)/2 R(G,V_\lambda)\cong k[[t]]$, all coefficients 
of $h_\lambda(t)$ must be divisible by $2$. This proves Proposition \ref{prop:onetubes} when $i=2$.

Suppose now that $B$ is Morita equivalent to $\Lambda_{1,1}$. For $\lambda\in k$, let
$V_\lambda$ be the $B$-module which under the Morita equivalence corresponds to
the $\Lambda_{1,1}$-module $S^{(\lambda)}_{010}$ from Definition \ref{def:modules}.
Since by \cite[Prop. 3.1]{holm} and \cite[Cor. 5.5]{rickard1}, 
there is a stable equivalence of Morita type between $\Lambda_{1,1}$ and 
$\Lambda_{2,1}$, it follows by \cite[Lemma 2.2.3]{3sim} that 
$R(G,V_\lambda)/2R(G,V_\lambda)\cong k[[t]]$. 
Since $\mathrm{Ext}^2_{kG}(V_\lambda,V_\lambda)\cong k$, we can argue as in the case
when $i=2$ that $R(G,V_\lambda)\cong W[[t]]/(2\, f_\lambda(t))$ for a certain $f_\lambda(t)
\in W[[t]]$. This proves Proposition \ref{prop:onetubes} when $i=1$.
\end{proof}


\section{Background: Representation theory of the basic algebras from Section $\ref{ss:basicalgebras}$}
\label{s:stringband}

Let $i\in\{1,2\}$, let $c\in\{0,1\}$ and let $d\ge 3$.
In this section, we give a brief introduction into the representation theory of the basic
$k$-algebras $\Lambda_{i,c}$ from Section  \ref{ss:basicalgebras} (see Figure \ref{fig:basic}).
It follows from the definition of string algebras in \cite[Sect. 3]{buri} that 
$\Lambda_{i,c}/\mathrm{soc}(\Lambda_{i,c})$ is a string algebra.
Therefore, one can see as in \cite[Sect. I.8.11]{erd} that
the isomorphism classes of all non-projective indecomposable 
$\Lambda_{i,c}$-modules are given by string and band modules as defined in \cite[Sect. 3]{buri}.
Define
\begin{eqnarray}
\label{eq:j1}
J_1&=&\{\beta\gamma, \alpha^2,
(\gamma\beta\alpha)^{2^{d-2}},(\alpha\gamma\beta)^{2^{d-2}},(\beta\alpha\gamma)^{2^{d-2}}\}
\subset kQ_1,\\
\label{eq:j2}
J_2&=&\{\alpha^2,\eta\beta,\gamma\eta,\beta\gamma,
\gamma\beta\alpha,\alpha\gamma\beta,\beta\alpha\gamma,\eta^{2^{d-2}}\}\subset kQ_2.
\end{eqnarray}
Then $\Lambda_{i,c}/\mathrm{soc}(\Lambda_{i,c})=
kQ_i/\langle J_i\rangle$.


\subsection{String and band modules}
\label{ss:strings}

For each arrow $\zeta$ in $Q_i$, we define a formal inverse
$\zeta^{-1}$ with starting point $s(\zeta^{-1})=e(\zeta)$ and end point $e(\zeta^{-1})=
s(\zeta)$.
A word relative to the quiver $Q_i$ is a sequence $w=w_1\cdots w_n$, where each $w_j$ is either an 
arrow or a formal inverse in $Q_i$ such that $s(w_j)=e(w_{j+1})$ for $1\leq j \leq n-1$. Define 
$s(w)=s(w_n)$, $e(w)=e(w_1)$ and 
$w^{-1}=w_n^{-1}\cdots w_1^{-1}$. There are also empty words $1_0$ and $1_1$ of length $0$ with 
$e(1_0)=0=s(1_0)$, $e(1_1)=1=s(1_1)$ and $(1_0)^{-1}=1_0$, $(1_1)^{-1}=1_1$. 
Denote the set of all words relative to $Q_i$ by $\mathcal{W}_i$, and the set of all 
non-empty words $w$ with $e(w)=s(w)$ by ${\mathcal{W}}_i^r$. 

\begin{dfn}
\label{def:strings}
Let $\sim_s$ be the equivalence relation on $\mathcal{W}_i$ with $w\sim_s w'$ if and only if $w=w'$ or 
$w^{-1}=w'$. Then a \emph{string} for $\Lambda_{i,c}$ is a representative $w\in\mathcal{W}_i$ of 
an equivalence class under 
$\sim_s$ with the following property: Either $w=1_u$ for $u\in\{0,1\}$, or $w=w_1\cdots w_n$ where 
$w_j\neq w_{j+1}^{-1}$ for $1\leq j\leq n-1$ and no subword of $w$ or its formal inverse belongs to 
$J_i$.

Let $C=w_1\cdots w_n$ be such a string of length $n\ge 1$. Then there exists an
indecomposable $\Lambda_{i,c}$-module $M(C)$,
called the \emph{string module} $M(C)$ corresponding to the string $C$, which can be described as follows.
There is a $k$-basis $\{z_0,z_1,\ldots, z_n\}$ of $M(C)$ such that the action of 
$\Lambda_{i,c}$ on $M(C)$ 
is given by the following representation $\varphi_C:\Lambda_{i,c}\to\mathrm{Mat}(n+1,k)$. 
Let $v(j)=e(w_{j+1})$ for $0\leq j\leq n-1$ and $v(n)=s(w_n)$. Then for each vertex $u\in\{0,1\}$, 
for each arrow $\zeta$ in $Q_i$ and for each $0\le j\le n$,
$$\varphi_C(u)(z_j) = \left\{ \begin{array}{c@{\quad,\quad}l}
z_j & \mbox{if $v(j)=u$}\\ 0 & \mbox{else} \end{array} \right\}\; \mbox{ and } \;
\varphi_C(\zeta)(z_j) = \left\{ \begin{array}{c@{\quad,\quad}l}
z_{j-1} & \mbox{if $w_j=\zeta$}\\ z_{j+1} & \mbox{if $w_{j+1}=\zeta^{-1}$}\\
0 & \mbox{else}
\end{array} \right\} .$$
We call $\varphi_C$ the canonical representation and $\{z_0,z_1,\ldots,z_n\}$ a \emph{canonical 
$k$-basis} for $M(C)$ relative to the representative $C$. Note that $M(C)\cong M(C^{-1})$. 

The string modules for the empty words are isomorphic to the simple $\Lambda_{i,c}$-modules,
namely $M(1_0)\cong S_0$ and $M(1_1)\cong S_1$.
\end{dfn}

\begin{dfn}
\label{def:bands}
Let $w=w_1\cdots w_n\in {\mathcal{W}}_i^r$. Then, for $0\leq j\leq n-1$, the $j$-th rotation of $w$ is 
defined to be the word $\rho_j(w)=w_{j+1}\cdots w_n w_1 \cdots w_j$. Let $\sim_r$ be the 
equivalence relation on ${\mathcal{W}}_i^r$ such that
$w\sim_r w'$ if and only if $w=\rho_j(w')$ for some $j$ or $w^{-1}=\rho_j(w')$ for some $j$. 
Then a \emph{band} for $\Lambda_{i,c}$ is a representative $w\in {\mathcal{W}}_i^r$ of 
an equivalence class under $\sim_r$ with the following property: 
$w=w_1\cdots w_n$, $n\ge 1$, with $w_j\neq w_{j+1}^{-1}$ and $w_n\neq w_1^{-1}$, such that 
$w$ is not a power of a smaller word, and, for all positive integers $m$, no subword of $w^m$ 
or its formal inverse belongs to $J_i$.

Let $B=w_1\cdots w_n$ be such a band of length $n$. Then for each integer $m>0$ and each 
$\lambda\in k^*$ there exists an indecomposable $\Lambda_{i,c}$-module $M(B,\lambda,m)$ which is
called the \emph{band module} corresponding to the band $B$, $\lambda$ and $m$, which can be described
as follows. There is a $k$-basis 
$$\{z_{0,1},z_{0,2},\ldots, z_{0,m},z_{1,1},\ldots,z_{1,m},\ldots, z_{n-1,1},\ldots,z_{n-1,m}\}$$
of $M(B,\lambda,m)$ such that the action of $\Lambda_{i,c}$ on $M(B,\lambda,m)$ 
is given by the following representation $\varphi_{B,\lambda,m}:\Lambda_{i,c}\to\mathrm{Mat}(n\cdot m,k)$. 
Let $v(j)=e(w_{j+1})$ for $0\leq j\leq n-1$. Then for each vertex $u\in\{0,1\}$, for 
each arrow $\zeta$ in $Q_i$ and for all $0\le j\le n-1$, $1\le j' \le m$, we have
\begin{eqnarray*}
\varphi_{B,\lambda,m}(u)(z_{j,j'}) &=& \left\{ \begin{array}{c@{\quad,\quad}l}
z_{j,j'} & \mbox{if $v(j)=u$}\\ 0 & \mbox{else} \end{array} \right\}\; \mbox{ and } \;\\[1ex]
\varphi_{B,\lambda,m}(\zeta)(z_{j,j'})& =& \left\{ \begin{array}{c@{\quad,\quad}l}
\lambda\, z_{0,j'}+z_{0,j'+1} & \mbox{if $w_j=\zeta$ and $j= 1$}\\
z_{j-1,j'} & \mbox{if $w_j=\zeta$ and $j\neq 1$}\\ 
\lambda^{-1} z_{1,j'}+z_{1,j'+1}&  \mbox{if $w_{j+1}=\zeta^{-1}$ and $j=0$}\\
z_{j+1,j'} & \mbox{if $w_{j+1}=\zeta^{-1}$ and $j\neq 0$}\\
0 & \mbox{else}
\end{array} \right\} 
\end{eqnarray*}
where $z_{0,m+1}=0=z_{1,m+1}$ and $z_{n,j'}=z_{0,j'}$ for all $j'$.
We  call $\varphi_{B,\lambda,m}$ the canonical representation and 
$\{z_{0,1},z_{0,2},\ldots, z_{0,m},z_{1,1},\ldots,z_{1,m},\ldots, z_{n-1,1},\ldots,z_{n-1,m}\}$ 
a \emph{canonical $k$-basis} for $M(B,\lambda,m)$ relative to the representative $B$.
Note that we have for all $0\le j\le n-1$,
$$M(B,\lambda,m)\cong M(\rho_j(B),\lambda,m)\cong M(\rho_j(B)^{-1},\lambda^{-1},m).$$ 
\end{dfn}


\subsection{The stable Auslander-Reiten quiver}
\label{ss:stablearquiver}

Each component of the stable Auslander-Reiten quiver of $\Lambda_{i,c}$ consists either entirely of 
string modules or entirely of band modules. The band modules all lie in $1$-tubes. The components
consisting of string modules are one $1$-tube, one $3$-tube and infinitely many non-periodic 
components of type $\mathbb{Z}A_\infty^\infty$. 
The irreducible morphisms between string modules can be described using hooks and cohooks,
which are defined as follows. Let 
\begin{eqnarray}
\label{eq:m1}
 \mathcal{M}_1&=&\{\beta\alpha(\gamma\beta\alpha)^{2^{d-2}-1},
 \gamma\beta(\alpha\gamma\beta)^{2^{d-2}-1},
 \alpha\gamma(\beta\alpha\gamma)^{2^{d-2}-1},1_1\}
\subset kQ_1,\\
\label{eq:m2}
\mathcal{M}_2&=&\{\gamma\beta,\beta\alpha,\alpha\gamma,\eta^{2^{d-2}-1}\}\subset kQ_2.
\end{eqnarray}

\begin{dfn}
\label{def:arcomps}
Let $S$ be a string for $\Lambda_{i,c}$.
We say that $S$ starts on a peak (resp. starts in a deep) if for each arrow $\zeta$ in $Q_i$,
$S\zeta$ (resp. $S\zeta^{-1}$) is not a string.
Dually, we say that $S$ ends on a peak (resp. ends in a deep) if for each arrow $\xi$ in $Q_i$,
$\xi^{-1}S$ (resp. $\xi S$) is not a string.

If $S$ does not start on a peak (resp. does not start in a deep), there is a unique arrow $\zeta$ 
in $Q_i$ and a unique $M\in\mathcal{M}_i$ such that $S_h=S\zeta M^{-1}$ 
(resp. $S_c=S\zeta^{-1}M$) is a string. 
We say $S_h$ (resp. $S_c$) is obtained from $S$ by adding a \emph{hook} (resp. a \emph{cohook}) on 
the right side.

Dually, if $S$ does not end on a peak (resp. does not end in a deep), there is a unique arrow $\xi$ 
in $Q_i$ and a unique $N\in\mathcal{M}_i$ such that ${}_hS=N\xi^{-1}S$ 
(resp. ${}_cS=N^{-1}\xi S$) is a string. 
We say ${}_hS$ (resp. ${}_cS$) is obtained from $S$ by adding a \emph{hook} (resp. a \emph{cohook}) 
on the left side.
\end{dfn}

All irreducible morphisms between string modules are either canonical injections
$M(S)\to M(S_h)$, $M(S)\to M({}_hS)$,
or canonical projections
$M(S_c)\to M(S)$, $M({}_cS)\to M(S)$.


\subsection{Homomorphisms between string and band modules}
\label{ss:homs}

In \cite{krau}, all homomorphisms between string and band modules were determined.
The following remark describes the homomorphisms between string modules using the
canonical bases defined in Definition \ref{def:strings}.

\begin{rem}
\label{rem:stringhoms}
Let $S$ and $T$ be strings for $\Lambda_{i,c}$. Let
$M(S)$ (resp. $M(T)$) be the corresponding string module with a
canonical $k$-basis $\{x_u\}_{u=0}^m$ (resp. $\{y_v\}_{v=0}^n$) 
relative to the representative $S$ (resp. $T$).
Suppose $C$ is a string for $\Lambda_{i,c}$ such that
\begin{enumerate}
\item[(i)] $S\sim_s S'CS''$ with ($S'$ of length $0$ or $S'=\hat{S}'\zeta_1$) and ($S''$ of length $0$ or 
$S''=\zeta_2^{-1} \hat{S}''$), where $S',\hat{S}',S'',\hat{S}''$ are strings relative to $Q_i$ and $J_i$
and $\zeta_1,\zeta_2$ are arrows in $Q_i$; 
and
\item[(ii)] $T\sim_sT'CT''$ with ($T'$ of length $0$ or $T'=\hat{T}'\xi_1^{-1} $) and ($T''$ of length $0$ or 
$T''=\xi_2 \hat{T}''$), where $T',\hat{T}',T'',\hat{T}''$ are strings relative to $Q_i$ and $J_i$
and $\xi_1,\xi_2$ are arrows in $Q_i$.
\end{enumerate} 
Then there exists a non-zero $\Lambda_{i,c}$-module homomorphism $\sigma_C:M(S)\to M(T)$ which factors through $M(C)$ and which sends 
each element of $\{x_u\}_{u=0}^m$ either to zero or to an element of $\{y_v\}_{v=0}^n$,
according to the relative position of $C$ in $S$ and $T$, respectively.
If e.g. $S=s_1s_2\cdots s_m$, $T=t_1t_2\cdots t_n$, and $C=s_{a+1}s_{a+2}\cdots s_{a+\ell} = 
t_{b+\ell}^{-1}t_{b+\ell-1}^{-1}\cdots t_{b+1}^{-1}$,
then 
$$\sigma_C(x_{a+t})=y_{b+\ell-t} \mbox{ for } 0\le t\le \ell, \mbox{ and } \sigma_C(x_u)=0\mbox{ for all 
other $u$.}$$
Note that there may be several choices of $S',S''$ (resp. $T',T''$) in (i) (resp. (ii)). In other words, there 
may be several $k$-linearly independent homomorphisms factoring through $M(C)$. By 
\cite{krau}, 
every $\Lambda_{i,c}$-module homomorphism $\sigma:M(S)\to M(T)$ is a unique
$k$-linear combination of 
homomorphisms which factor through string modules corresponding to strings $C$ satisfying 
(i) and (ii). 
\end{rem}

\begin{rem}
\label{rem:bandhoms}
Suppose $B$ is a band for $\Lambda_{i,c}$, $\lambda\in k^*$ and $m$ is a positive integer.
It follows from \cite{krau} that if $m\ge 2$, then the stable endomorphism ring
$\underline{\mathrm{End}}_{\Lambda_{i,c}}(M(B,\lambda,m))$ has $k$-dimension at least $2$. Moreover, if $m=1$, then the endomorphism ring of 
$M(B,\lambda,1)$ can be described in a similar way as in Remark \ref{rem:stringhoms}, using the
canonical bases defined in Definition \ref{def:bands}.
\end{rem}



\begin{thebibliography}{88}

\bibitem{alp} J.~L.~Alperin, Local representation theory. Modular representations as an introduction to 
the local representation theory of finite groups. Cambridge Studies in Advanced Mathematics, vol. 11, 
Cambridge University Press, Cambridge, 1986.

\bibitem{ben} D.~J.~Benson, Representations and Cohomology, I. Cambridge Studies in Advanced 
Mathematics, vol.  30, Cambridge University Press, Cambridge, 1991.

\bibitem{bl} F.~M.~Bleher, Universal deformation rings and Klein four defect groups. Trans. Amer. Math. Soc. 354 (2002), 3893--3906.
	
\bibitem{diloc} F.~M.~Bleher, Universal deformation rings for dihedral $2$-groups. 
J. London Math. Soc. (2) 79 (2009), 225--237. 

\bibitem{3sim} F.~M.~Bleher, Universal deformation rings and dihedral defect groups. 
Trans. Amer. Math. Soc. 361 (2009),  3661--3705. 

\bibitem{parameter} F.~M.~Bleher, Dihedral blocks with two simple modules.
Proc. Amer. Math. Soc. 138 (2010), 3467--3479.

\bibitem{quaternion} F.~M.~Bleher, Universal deformation rings and generalized quaternion
defect groups. Adv. Math. 225 (2010), 1499--1522.

\bibitem{bc} F.~M.~Bleher and T.~Chinburg, Universal deformation rings and cyclic blocks. Math. Ann. 318 (2000), 805--836.

\bibitem{bc4.9} F.~M.~Bleher and T.~Chinburg, Universal deformation rings need not be
complete intersections. C. R. Math. Acad. Sci. Paris 342 (2006),  229--232.

\bibitem{bc5} F.~M.~Bleher and T.~Chinburg, Universal deformation rings need not be complete 
intersections. Math. Ann. 337 (2007), 739--767. 

\bibitem{blello} F.~M.~Bleher and G.~Llosent, Universal deformation rings for the symmetric group 
$S_4$. Algebr. Represent. Theory 13 (2010), 255--270.

\bibitem{brauer2} R.~Brauer, On $2$-blocks with dihedral defect groups. Symposia Mathematica, vol. XIII (Convegno di Gruppi e loro Rappresentazioni, INDAM, Rome, 1972), pp. 367--393, Academic Press, London, 1974. 

\bibitem{breuil} C.~Breuil, B.~Conrad, F.~Diamond and R.~Taylor, On the modularity of 
elliptic curves over $\mathbb{Q}$: 
Wild $3$-adic exercises. J. Amer. Math. Soc. 14 (2001), 843--939.
    
\bibitem{buri} M.~C.~R.~Butler and C.~M.~Ringel, Auslander-Reiten sequences with few middle terms and applications to string algebras. Comm. Algebra 15 (1987), 145--179. 

\bibitem{carl2} J.~F.~Carlson and J.~Th\'{e}venaz, The classification of endo-trivial modules.  
Invent. Math.  158  (2004),  389--411.

\bibitem{flach} T.~Chinburg, Can deformation rings of group representations not be local complete intersections? In: Problems from the Workshop on Automorphisms of Curves.  Edited by Gunther Cornelissen and Frans Oort, with contributions by I. Bouw, T. Chinburg, G. Cornelissen, C. Gasbarri, D. Glass, C. Lehr, M. Matignon, F. Oort, R. Pries and S. Wewers. Rend. Sem. Mat. Univ. Padova 113 (2005), 129--177.

\bibitem{cornell} G.~Cornell, J.~H.~Silverman and G.~Stevens (eds.), Modular Forms and Fermat's Last Theorem (Boston, 1995). Springer-Verlag, Berlin-Heidelberg-New York, 1997.

\bibitem{CR} C.~W.~Curtis and I.~Reiner, Methods of representation theory. Vol. I. With applications to 
	finite groups and orders. John Wiley and Sons, Inc., New York, 1981. 

\bibitem{lendesmit} B.~de Smit and H.~W.~Lenstra, Explicit construction of universal deformation rings. In: Modular Forms and Fermat's Last Theorem (Boston, MA, 1995), Springer-Verlag, Berlin-Heidelberg-New York, 1997, pp. 313--326.

\bibitem{erd} K.~Erdmann, Blocks of Tame Representation Type and Related Algebras, Lecture Notes in Mathematics, vol. 1428, Springer-Verlag, Berlin-Heidelberg-New York, 1990.

\bibitem{erdlater} K.~Erdmann, On 2-modular representations of $GU_2(q)$, $q \equiv 3 \mod 4$, Comm. Algebra 20 (1992), 3479--3502.

\bibitem{fong} P.~Fong, A note on splitting fields of representations of finite groups.  
Illinois J. Math.  7  (1963) 515--520.

\bibitem{gowa} D.~Gorenstein and J.~H.~Walter, The characterization of finite groups with dihedral Sylow $2$-subgroups. I, II, III. J. Algebra 2 (1965) 85--151, 218--270, 354--393.

\bibitem{holm} T.~Holm, Derived Equivalence Classification of Algebras of 
Dihedral, Semidihedral, and Quaternion Type. J. Algebra 211 (1999), 159--205. 

\bibitem{hup} B.~Huppert, Endliche Gruppen. I. Die Grundlehren der Mathematischen Wissenschaften, Band 134, Springer-Verlag, Berlin-New York, 1967.


\bibitem{krau} H.~Krause, Maps between tree and band modules. J. Algebra 137 (1991), 186--194.

\bibitem{linckel} M.~Linckelmann, A derived equivalence for blocks with dihedral defect groups. J. 
Algebra 164 (1994), 244--255.

\bibitem{linckel1} M.~Linckelmann, The source algebras of blocks with a Klein four defect group, J.
Algebra 167 (1994), 821--854.
       
\bibitem{maz1} B.~Mazur, Deforming Galois representations. In: Galois groups over $\mathbb{Q}$ (Berkeley, CA, 1987), Springer-Verlag, Berlin-Heidelberg-New York, 1989, pp. 385--437.

\bibitem{rickard1} J.~Rickard, Derived equivalences as derived  functors.
        J. London Math. Soc. 43 (1991), 37--48.
        
\bibitem{taywiles} R.~Taylor and A.~Wiles, Ring-theoretic properties of certain Hecke algebras. Ann. of Math. 141 (1995), 553--572.

\bibitem{wiles} A.~Wiles, Modular elliptic curves and Fermat's last theorem.  Ann. of Math. 141 (1995), 443--551.

\end{thebibliography}
\end{document}